\documentclass[12pt]{amsart}
\usepackage{amsmath}
\usepackage{latexsym}
\usepackage{amssymb}
\usepackage{amscd}
\usepackage{amsfonts}
\usepackage[all]{xy}

\usepackage{tabularx}
\newcolumntype{Y}{>{\small\raggedright\arraybackslash}X}

\theoremstyle{definition}

\newtheorem{thm}{Theorem}[section]

\newtheorem{defi}[thm]{Definition}
\newtheorem{prop}[thm]{Proposition}
\newtheorem{cor}[thm]{Corollary}
\newtheorem{lem}[thm]{Lemma}
\newtheorem{rmk}[thm]{Remark}

\newtheorem*{notations}{Notations}
\newtheorem*{acknowledgement}{Acknowledgement}

\theoremstyle{remark}

\begin{document}

\title[Exceptional sequences on some surfaces isogenous to a higher product]{Exceptional sequences of maximal length on some surfaces isogenous to a higher product}

\author{Kyoung-Seog Lee}

\address{Department of Mathematics, Seoul National University, Seoul 151-747, Korea}

\email{kyoungseog02@gmail.com}


\begin{abstract}
Let $S=(C\times D)/G$ be a surface isogenous to a higher product of unmixed type with $p_g=q=0$, $G=(\mathbb{Z}/2)^3$ or $(\mathbb{Z}/2)^4$. We construct exceptional sequences of line bundles of maximal length and quasiphantom categories on $S$.
\end{abstract}

\maketitle

\section*{Introduction}

Derived category of an algebraic variety is an interesting invariant
containing much information about the variety. Algebraic varieties
having equivalent derived categories share many geometric properties
\cite{Huy}.

One of the most powerful tools to study derived categories is the notion of
semiorthogonal decomposition. A semiorthogonal
decomposition divides a derived category into simpler
subcategories and we can study the derived category via these simpler subcategories. One way to get a semiorthogonal decomposition is to construct an exceptional sequence. When we have an exceptional sequence then we get an admissible triangulated subcategory generated by the exceptional sequence and its orthogonal complement which give a semiorthogonal decomposition.

There are lots of studies about semiorthogonal decompositions of
derived categories of smooth projective varieties, especially for
rational or Fano varieties. Many rational varieties have exceptional
sequences which generate the derived categories of them. Especially
every smooth projective rational surface has an exceptional sequence
which generates its derived category \cite{King},\cite{Or}. It is
known that every toric variety also has an exceptional sequence which
generates its derived category \cite{Kaw}. For a Fano variety, the
structure sheaf is an exceptional object and there exist at least
one semiorthogonal decomposition \cite{Ku2}.

It is expected that the behaviours of derived categories of varieties
with nonnegative Kodaira dimensions will be very different from those
of rational or Fano varieties. For example it is known that there is
no nontrivial semiorthogonal decomposition for curves with genus
greater than or equal to 1 \cite{Ok} or varieties having trivial
canonical bundles. In particular they do not have any exceptional
object.

For a surface of general type with $p_g=q=0$, the structure
sheaf is an exceptional object and there is already a semiorthogonal
decomposition. If there is another exceptional object in the orthogonal complement of the structure sheaf then we can divide the orthogonal complement into two smaller pieces. Then it is an interesting question how much we
can extend the exceptional sequence in the derived category. It is
easy to show that the length of the exceptional sequence is bounded
by the rank of the Grothendieck group. When the length of
exceptional sequence is the rank of Grothendieck group we call the
exceptional sequence is of maximal length.

Recently there are some constructions of exceptional sequences of
maximal lengths on surfaces of general type with $p_g=q=0$. In \cite{AO}, \cite{BBKS}, \cite{BBS}, \cite{C}, \cite{F}, \cite{GKMS}, \cite{GS}, \cite{Lee} the authors
constructed exceptional sequences of maximal lengths consisting of
line bundles on surfaces of general type with $p_g=q=0$. The
triangulated subcategories generated by exceptional sequences are
not full in these cases. The categories of orthogonal complements of
these exceptional sequences have vanishing Hochschild homologies and
finite Grothendieck groups. They are called the quasiphantom
categories. It seems that these semiorthogonal decompositions
contain much information about the geometry of surfaces of general
type and may provide lots of unexpected feature of the derived
categories of algebraic varieties. For example in \cite{BBS} the authors constructed the first counterexample to the nonvanishing conjecture, and in \cite{BBS2} the authors gave the first counterexample to geometric Jordan-H\"{o}lder property using the construction of \cite{BBS} (see also \cite{Ku1}).

Let $S=(C \times D)/G$ be a surface isogenous to a higher product of unmixed type with $p_g=q=0$. When $G$ is abelian, Bauer and Catanese proved there are 4 possible groups, $(\mathbb{Z}/2)^3, (\mathbb{Z}/2)^4, (\mathbb{Z}/3)^2, (\mathbb{Z}/5)^2$, and described their moduli spaces in \cite{BC}. Galkin and Shinder constructed exceptional sequences of maximal length on the Beauville surface, the $(\mathbb{Z}/5)^2$ case in \cite{GS}. Motivated by their work we constructed exceptional sequences of maximal length on the surfaces isogenous to a higher product of unmixed type with $p_g=q=0$ and $G=(\mathbb{Z}/3)^2$ in \cite{Lee}. Therefore it is a natural question whether the other surfaces isogenous to a higher product of unmixed type with $p_g=q=0$ admit exceptional sequences of maximal length when $G$ is $(\mathbb{Z}/2)^3$, or $(\mathbb{Z}/2)^4$.

In this paper we construct exceptional sequences of line bundles of maximal length on surfaces isogenous to a higher product of unmixed type with $p_g=q=0$, $G=(\mathbb{Z}/2)^3$ or $G=(\mathbb{Z}/2)^4$.

\begin{thm} Let $S=(C \times D)/G$ be a surface isogenous to a higher product of unmixed type with $p_g=q=0$, $G=(\mathbb{Z}/2)^3$ or $G=(\mathbb{Z}/2)^4$. There are exceptional sequences of line bundles of maximal length on $S$. The orthogonal complements of the admissible subcategories generated by these exceptional sequences in the derived category of $S$ are quasiphantom categories.
\end{thm}

The idea of the constructions are as follows. Let $S=(C \times D)/G$
be a surface isogenous to a higher product of unmixed type with
$p_g=q=0$. In the previous paper \cite{Lee}, we proved that one
cannot construct any exceptional sequence of maximal length
consisting of line bundles on $S$ using $G$-equivariant line bundles
on $C$ and $D$ when $G=(\mathbb{Z}/2)^3$ or $G=(\mathbb{Z}/2)^4$. In
this paper we use $G$-invariant line bundles instead of $G$-equivariant line
bundles on $C$ and $D$ to construct exceptional sequences of maximal
length on $S$. To be more precise, we show that there are $G$-invariant theta
characteristics and $G$-invariant torsion line bundles on $C$ and
$D$ whose box products become $G$-equivariant line bundles on $C
\times D$. To do this we explicitly compute the Schur multipliers
of these invariant line bundles in the cocyle level. We show that one can
produce exceptional sequences of maximal length on $S$ by this way.
We expect that one can construct exceptional sequences of maximal
lengths on surfaces isogenous to a higher product with nonabelian
quotient groups in a similar way.

We also compute the Hochschild cohomologies of the quasiphantom
categories and prove that for some exceptional sequences we made the
DG algebras of endomorphisms are deformation invariant.

\begin{acknowledgement}
I am grateful to my advisor Young-Hoon Kiem for his invaluable
advice and many suggestions for the first draft of this paper.
Without his support and encouragement, this work could not have been
accomplished. I thank Fabrizio Catanese, Igor Dolgachev for answering my questions, and helpful conversations. I would like to thank Seoul National University for its support during the preparation of this paper.
\end{acknowledgement}

\begin{notations}
We will work over $\mathbb{C}$. A curve will mean a smooth projective curve. A surface will mean a smooth projective surface. Derived category of a variety will mean the bounded derived category of coherent sheaves on the variety. $G$ denotes a finite group and $\widehat{G}=Hom(G,\mathbb{C}^*)$ denotes the character group of $G$. Here $ \sim $ denotes linear equivalence of divisors.
\end{notations}

\section{Surfaces isogenous to a higher product}

In this section we recall the definition and some basic facts about surfaces isogenous to a higher product. For details, see \cite{BC}.

\begin{defi} A surface $S$ is called isogenous to a higher product if $S=(C \times D)/G$ where $C$, $D$ are curves with genus at least $2$ and $G$ is a finite group acting freely on $C \times D$. When $G$ acts via a product action, $S$ is called of unmixed type.
\end{defi}

\begin{rmk} \cite{BC} Let $S$ be a surface isogenous to a higher product of unmixed type. Then $S$ is a surface of general type. When $p_g=q=0$, one can prove that $K_S^2=8$, $C/G \cong D/G \cong \mathbb{P}^1$ and $|G|=(g_C-1)(g_D-1)$ where $g_C$ and $g_D$ denote the genus of $C$ and $D$, respectively.
\end{rmk}

Bauer and Catanese proved that there are four families of surfaces isogenous to a higher product of unmixed type with $p_g=q=0$, $G$ is abelian. Moreover they computed the dimensions of the families they form in \cite{BC}.

\begin{thm} \cite{BC} Let $S$ be a surface isogenous to a higher product $(C \times D) /G$ of unmixed type with $p_g=q=0$. If $G$ is abelian, then $G$ is one of the following groups : \\
(1) $(\mathbb{Z}/2)^3$, and these surfaces form an irreducible connected component of dimension 5 in their moduli space; \\
(2) $(\mathbb{Z}/2)^4$, and these surfaces form an irreducible connected component of dimension 4 in their moduli space; \\
(3) $(\mathbb{Z}/3)^2$, and these surfaces form an irreducible connected component of dimension 2 in their moduli space; \\
(4) $(\mathbb{Z}/5)^2$, and $S$ is the Beauville surface.
\end{thm}

Recently Shabalin \cite{Sh} and Bauer, Catanese and Frapporti \cite{BC2}, \cite{BCF} have computed the first homology groups of these surfaces.

\begin{thm} \cite{BC2}, \cite{BCF}, \cite{Sh} Let $S$ be a surface isogenous to a higher product $(C \times D)/G$ of unmixed type with $p_g=q=0$, $G$ is an abelian group. Then we have the following isomorphisms : \\
(1) $H_1(S,\mathbb{Z}) \cong (\mathbb{Z}/2)^4 \oplus (\mathbb{Z}/4)^2$ for $G=(\mathbb{Z}/2)^3$; \\
(2) $H_1(S,\mathbb{Z}) \cong (\mathbb{Z}/4)^4$ for $G=(\mathbb{Z}/2)^4$; \\
(3) $H_1(S,\mathbb{Z}) \cong (\mathbb{Z}/3)^5$ for $G=(\mathbb{Z}/3)^2$; \\
(4) $H_1(S,\mathbb{Z}) \cong (\mathbb{Z}/5)^3$ for $G=(\mathbb{Z}/5)^2$.
\end{thm}

\begin{rmk}
Let $S$ be a surface with $p_g=q=0$ and $K_S^2=8$. From the exponential sequence $$ 0 \to \mathbb{Z} \to \mathcal{O} \to \mathcal{O}^* \to 0 $$ we get $$ Pic(S) \cong H^2(S,\mathbb{Z}). $$

Because $q=0$ we get $b_1=0$. Noether's formula $$\chi(\mathcal{O}_X) = 1 =
\frac{1}{12}(8+2b_0-2b_1+b_2)=\frac{1}{12}(K_S^2 + \chi_{top}(S))$$
implies that these surfaces have $b_2=2$.
\end{rmk}

Now we compute the Grothendieck groups of these surfaces. 

\begin{lem} \cite[Lemma 2.7]{GS}, \cite[Lemma 2.6]{Lee} Let $S$ be a surface with $p_g=q=0$ isogenous to a higher product $(C \times D)/G$ of unmixed type and let $G$ be abelian. Then
$$ K(S) \cong \mathbb{Z}^2 \oplus Pic(S). $$
\end{lem}

\section{Derived categories of surfaces isogenous to a higher product of unmixed type with $p_g=q=0$, $G=(\mathbb{Z}/3)^2$ or $G=(\mathbb{Z}/5)^2$}

We recall some basic notions to describe derived categories of algebraic varieties.

\begin{defi}
(1) An object $E$ in a triangulated category $D$ is called exceptional if
\begin{displaymath}
Hom(E,E[i])=\left \{ {\begin{array}{ll} \mathbb{C} & \textrm{if $i=0$,} \\ 0 & \textrm{otherwise.} \end{array}}
\right.
\end{displaymath}
(2) A sequence $E_1, \cdots, E_n$ of exceptional objects is called
an exceptional sequence if $$Hom(E_i,E_j[k])=0, \forall i > j,
\forall k.$$
\end{defi}

When $S$ is a surface with $p_g=q=0$, every line bundle on $S$ is an
exceptional object in $D^b(S)$. Now we define the notion of
semi-orthogonal decomposition.

\begin{defi}\cite{Huy}
Let $D$ be a triangulated category. \\
(1) A full triangulated subcategory $D' \subset D$ is called admissible if the inclusion has a right adjoint. \\
(2) A sequence of full admissible triangulated subcategories $D_1, \cdots, D_n \subset D$ is semi-orthogonal if for all $i > j$ $$ D_j \subset D_i^{\perp}. $$
(3) Such a sequence is called a semi-orthogonal decomposition if $D_1, \cdots, D_n$ generate $D$.
\end{defi}

Next we define the notion of quasiphantom category.

\begin{defi}\cite[Definition 1.8]{GO}
Let $X$ be a smooth projective variety. Let $\mathcal{A}$ be an admissible triangulated subcategory of $D^b(X)$. Then $\mathcal{A}$ is called a quasiphantom category if the Hochchild homology of $\mathcal{A}$ vanishes, and the Grothendieck group of $\mathcal{A}$ is finite. If the Grothendieck group of $\mathcal{A}$ also vanishes, then $\mathcal{A}$ is called a phantom category.
\end{defi}

Now we discuss about previous construction of exceptional sequences of maximal length on the surfaces isogenous to a higher product of unmixed type with $p_g=q=0$, $G=(\mathbb{Z}/3)^2$ or $G=(\mathbb{Z}/5)^2$. When $G=(\mathbb{Z}/5)^2$, the surface isogenous to a higher product is called the Beauville surface. Galkin and Shinder constructed exceptional sequences of maximal length on the Beauville surface in \cite{GS}. Motivated by their work, we constructed exceptional sequences of maximal length on the surfaces isogenous to a higher product when $G=(\mathbb{Z}/3)^2$ in \cite{Lee}. Recently Coughlan has constructed exceptional sequences of maximal length on these surfaces via different approach in \cite{C}.

\begin{thm}\cite{C}, \cite{GS}, \cite{Lee} Let $S=(C \times D)/G$ be a surface isogenous to a higher product of unmixed type with $p_g=q=0$, $G=(\mathbb{Z}/3)^2$ or $G=(\mathbb{Z}/5)^2$. There are exceptional sequences of line bundles of maximal length on $S$. The orthogonal complements of the admissible subcategories generated by these exceptional sequences in the derived category of $S$ are quasiphantom categories.
\end{thm}

In this paper we will show that similar statements are true for surfaces isogenous to a higher product of unmixed type with $p_g=q=0$, $G=(\mathbb{Z}/2)^3$ or $G=(\mathbb{Z}/2)^4$.

\section{Theta characteristics}

In this section we collect some facts about $\mathbb{Z}/2$-invariant line bundles on curves studied by Beauville in {\cite{Be2}}. Let $C$ be a curve with involution $\sigma$, $B$ be the quotient curve $C/\sigma$, $\pi : C \to B$ be the quotient map, and $R \subset C$ be the set of ramification points. The double covering corresponds to a line bundle $\rho$ on $B$ such that $\rho^2=\mathcal{O}_B(\pi_*R)$, see \cite{BHPV}, \cite{Be1}, \cite{Be2}.

Beauville classifies all $\sigma$-invariant line bundles on $C$ in \cite{Be2}.

\begin{lem}{\cite[Lemma 1]{Be2}}
Consider the map $\phi : \mathbb{Z}^R \to Pic(C)$ which maps $r \in R$ to $\mathcal{O}_C(r)$. Its image lies in the subgroup $Pic(C)^\sigma$ of $\sigma$-invariant line bundles. When $R \neq \emptyset$, $\phi$ induces a short exact sequence $$ 0 \to \mathbb{Z}/2 \to (\mathbb{Z}/2)^R \to Pic(C)^{\sigma}/\pi^*Pic(B) \to 0, $$
and the kernel is generated by $(1,\cdots,1).$
\end{lem}

Beauville also showed how to compute the cohomologies of these invariant line bundles.

\begin{prop}{\cite[Proposition 1]{Be2}}
Let $M$ be a $\sigma$-invariant line bundle on $C$. Then \\
(1) $M \cong \pi^*L(E)$ for some $L \in Pic(B)$ and $E \subset R$. Any pair $(L',E')$ satisfying $M \cong \pi^*L'(E')$ is equal to $(L,E)$ or $(L \otimes \rho^{-1}(\pi_*E),R-E)$. \\
(2) There is a natural isomorphism $H^0(C,M) \cong H^0(B,L) \oplus H^0(B,L \otimes \rho^{-1}(\pi_*E)).$
\end{prop}

When the invariant line bundle is a theta characteristic, the above
proposition becomes as follows.

\begin{prop}{\cite[Proposition 2]{Be2}}
Let $\kappa$ be a $\sigma$-invariant theta characteristic on $C$. Then \\
(1) $\kappa \cong \pi^*L(E)$ for some $L \in Pic(B)$ and $E \subset R$ with $L^2 \cong K_B \otimes \rho(-\pi_*E).$ If another pair $(L',E')$ satisfies $\kappa \cong \pi^*L'(E')$, we have $(L',E')=(L,E)$ or $(L',E')=(K_B \otimes L^{-1},R-E)$. \\
(2) $h^0(\kappa) = h^0(L) + h^1(L)$.
\end{prop}

The above theorems are our main tools to compute the cohomologies of line bundles which we construct in this paper.

\section{$\alpha$-sheaves}

In this section we collect some facts in order to compute the
2-cocycles corresponding to the $G$-invariant line bundles on
curves. In this paper $G$ is an abelian group and $G$ acts on $\mathbb{C}^*$ trivially.
From this assumption the definition and computation of group cohomology become
much simpler than usual.

\subsection{generalities on group cohomology}

In 4.1 and 4.2 we recall some definitions and properties about group cohomology following \cite{Kar}.

\begin{defi}\cite{Kar} Let $G$ be an abelian group and $G$ acts on $\mathbb{C}^*$ trivially. \\
(1) A function $\alpha : \underbrace{G \times \cdots \times G}_n \to \mathbb{C}^*$ is called an n-cochain. \\
(2) An n-cochain $\alpha$ is called normalized if $\alpha(g_1,\cdots,g_n)=1$ whenever any of the $g_i=e \in G.$ \\
(3) A 2-cochain $\alpha$ is called 2-cocycle if $\alpha(x,y)\alpha(xy,z)=\alpha(y,z)\alpha(x,yz)$ for all $x,y,z \in G$, and we denote the abelian group of 2-cocycles by $Z^2(G,\mathbb{C}^*).$ \\
(4) A 2-cocycle $\alpha$ is called 2-coboundary if there exist a 1-cochain $t : G \to \mathbb{C}^*$ such that $\alpha(x,y) = t(x)t(y)t(xy)^{-1}$, and we denote the abelian group of 2-coboundaries by $B^2(G,\mathbb{C}^*)$. \\
(5) $H^2(G,\mathbb{C}^*) := Z^2(G,\mathbb{C}^*)/B^2(G,\mathbb{C}^*)$ is called the Schur multiplier of $G$.
\end{defi}

\subsection{Schur multiplier}

Now we focus on $H^2(G,\mathbb{C}^*)$. We only consider groups of
the form $G \cong (\mathbb{Z}/2)^r$ for some $r \in \mathbb{N}$.
Consider a decomposition of $G = N \times T \cong (\mathbb{Z}/2)^a
\times (\mathbb{Z}/2)^b$ for some $a,b \in \mathbb{N}$.

\begin{defi}
A cocycle $\alpha \in Z^2(G,\mathbb{C}^*)$ is normal if $\alpha(n,t)=1$, for all $n \in N, t \in T.$
\end{defi}

When $\alpha$ is a normal cycle then the following lemmas will enable us to compute $\alpha$.

\begin{lem}{\cite{Kar}}
(1) Each $\alpha \in Z^2(G,\mathbb{C}^*)$ is cohomologous to a normal cocycle $\beta$ such that $\alpha|_{T \times T}=\beta|_{T \times T}$. \\
(2) If $\alpha$ is a normal cocycle, then
$$ \alpha(nt,n't') = \alpha(t,t') \alpha(t,n') \alpha(n,n'), n.n' \in N, t,t' \in T.$$
In particular, a normal cocycle $\alpha$ is uniquely determined by
$$ \alpha|_{N \times N}, \alpha|_{T \times T}, \alpha|_{T \times N}. $$
\end{lem}

\begin{lem}{\cite{Kar}} $ \alpha|_{N \times N}, \alpha|_{T \times T}, \alpha|_{T \times N} $ determine a normal cocycle $\alpha$ of $Z^2(G,\mathbb{C}^*)$ if and only if the following four conditions hold: \\
(1) $\alpha|_{N \times N} \in Z^2(N,\mathbb{C}^*)$. \\
(2) $\alpha|_{T \times T} \in Z^2(T,\mathbb{C}^*)$. \\
(3) $\alpha(tt',n)=\alpha(t,n)\alpha(t',n), n \in N, t,t' \in T$. \\
(4) $\alpha(t,nn')=\alpha(t,n)\alpha(t,n'), n,n' \in N, t \in T$.
\end{lem}

\subsection{Generalities on $\alpha$-sheaves}

Let $X$ be an algebraic variety over $\mathbb{C}$, let $G$ be a
finite group acting on $X$, and let $\alpha$ be a 2-cocyle of $G$
with coefficients in $\mathbb{C}^*$. Elagin introduced the notion of
$\alpha$-sheaves, and proved some properties of them in \cite{E1},
\cite{E2}.

\begin{defi}\cite{E2} An $\alpha$-sheaf on $X$ is a coherent sheaf $F$ together with isomorphisms $\theta_g : g^*F \to F$ for all $g \in G$ such that $\theta_{gh} = \alpha(g,h) \theta_h \circ h^*\theta_g$ for any pair $g,h \in G$.
\end{defi}

\begin{rmk}
We can set $\alpha$ to be normalized canonically in our geometric
case.
\end{rmk}

\begin{prop}\cite[Proposition 1.2]{E2} The $\alpha$-sheaves on $X$ form abelian category. Let $\alpha$, $\beta$ be 2-cocycles of $G$. let $F$ and $G$ be $\alpha$- and $\beta$- sheaves on $X$. Then
$F \otimes G$ is an $\alpha \beta$ sheaf on $X$.
\end{prop}
\begin{proof}
$\theta_{gh} \otimes \theta_{gh} = \alpha(g,h) \beta(g,h) ( \theta_h \circ h^*\theta_g ) \otimes ( \theta_h \circ h^*\theta_g ) = \alpha(g,h) \beta(g,h) (\theta_h \otimes \theta_h) \circ h^*(\theta_g \otimes \theta_g)$ for any pair $g,h \in G.$
\end{proof}

\subsection{Basic example : $\mathcal{O}(1)$ bundle on $\mathbb{P}^1$ with $(\mathbb{Z}/2)^2$-action}

In this subsection we consider $\mathbb{P}^1$ with $G=(\mathbb{Z}/2)^2$-action and the $G$-invariant line bundle $\mathcal{O}(1).$

Let $G=(\mathbb{Z}/2)e_1 \oplus (\mathbb{Z}/2)e_2$, and $G$ acts on $\mathbb{P}^1$ as $e_1 \cdot [x:y]=[-x:y]$ and $e_2 \cdot [x:y]=[y:x]$. Let $U=Spec(\mathbb{C}[\frac{x}{y}])$ and $V=Spec(\mathbb{C}[\frac{y}{x}])$ be an affine open covering of $\mathbb{P}^1$ and let $\mathcal{O}(1) \cong \mathcal{O}([1:0])$ be a $G$-invariant line bundle. In this case we can compute $[\alpha] \in H^2(G,\mathbb{C}^*)$ by explicit computations as follows.

Fix four isomorphisms

$$ \theta_{0}(U) : 0^*\mathcal{O}([1:0])(U) \cong \mathbb{C}[\frac{x}{y}] \otimes_{\mathbb{C}[\frac{x}{y}]} \mathbb{C}[\frac{x}{y}] \to \mathbb{C}[\frac{x}{y}] \cong \mathcal{O}([1:0])(U), 1 \otimes 1 \mapsto 1, $$

$$ \theta_{0}(V) : 0^*\mathcal{O}([1:0])(V) \cong \frac{x}{y}\mathbb{C}[\frac{y}{x}] \otimes_{\mathbb{C}[\frac{y}{x}]} \mathbb{C}[\frac{y}{x}] \to \frac{x}{y}\mathbb{C}[\frac{y}{x}] \cong \mathcal{O}([1:0])(V), \frac{x}{y} \otimes 1 \mapsto \frac{x}{y}, $$

$$ \theta_{e_1}(U) : e_1^*\mathcal{O}([1:0])(U) \cong \mathbb{C}[\frac{x}{y}] \otimes_{\mathbb{C}[\frac{x}{y}]} \mathbb{C}[\frac{x}{y}] \to \mathbb{C}[\frac{x}{y}] \cong \mathcal{O}([1:0])(U), 1 \otimes 1 \mapsto 1, $$

$$ \theta_{e_1}(V) : e_1^*\mathcal{O}([1:0])(V) \cong \frac{x}{y} \mathbb{C}[\frac{y}{x}] \otimes_{\mathbb{C}[\frac{y}{x}]} \mathbb{C}[\frac{y}{x}] \to \frac{x}{y}\mathbb{C}[\frac{y}{x}] \cong \mathcal{O}([1:0])(V), \frac{x}{y} \otimes 1 \mapsto -\frac{x}{y}, $$

$$ \theta_{e_2}(U) : e_2^*\mathcal{O}([1:0])(U) \cong \frac{x}{y} \mathbb{C}[\frac{y}{x}] \otimes_{\mathbb{C}[\frac{y}{x}]} \mathbb{C}[\frac{x}{y}] \to \mathbb{C}[\frac{x}{y}] \cong \mathcal{O}([1:0])(U), \frac{x}{y} \otimes 1 \mapsto 1, $$

$$ \theta_{e_2}(V) : e_2^*\mathcal{O}([1:0])(V) \cong \mathbb{C}[\frac{x}{y}] \otimes_{\mathbb{C}[\frac{x}{y}]} \mathbb{C}[\frac{y}{x}] \to \frac{x}{y}\mathbb{C}[\frac{y}{x}] \cong \mathcal{O}([1:0])(V), 1 \otimes 1 \mapsto \frac{x}{y}, $$

$$ \theta_{e_1+e_2}(U) : (e_1+e_2)^*\mathcal{O}([1:0])(U) \cong \frac{x}{y} \mathbb{C}[\frac{y}{x}] \otimes_{\mathbb{C}[\frac{y}{x}]} \mathbb{C}[\frac{x}{y}] \to \mathbb{C}[\frac{x}{y}] \cong \mathcal{O}([1:0])(U), \frac{x}{y} \otimes 1 \mapsto -1, $$

$$ \theta_{e_1+e_2}(V) : (e_1+e_2)^*\mathcal{O}([1:0])(V) \cong \mathbb{C}[\frac{x}{y}] \otimes_{\mathbb{C}[\frac{x}{y}]} \mathbb{C}[\frac{y}{x}] \to \frac{x}{y}\mathbb{C}[\frac{y}{x}] \cong \mathcal{O}([1:0])(V), 1 \otimes 1 \mapsto \frac{x}{y}, $$

To compute $\alpha(e_1,e_2), \alpha(e_2,e_1)$ from
$$ \theta_{e_1+e_2}=\alpha(e_1,e_2)\theta_{e_2} \circ {e_2}^*\theta_{e_1}, $$
$$ \theta_{e_1+e_2}=\alpha(e_2,e_1)\theta_{e_1} \circ {e_1}^*\theta_{e_2}, $$
consider the following isomorphisms.

$$ e_2^*\theta_{e_1}(U) : e_2^*e_1^*\mathcal{O}([1:0])(U) \cong \frac{x}{y} \mathbb{C}[\frac{y}{x}] \otimes_{\mathbb{C}[\frac{y}{x}]} \mathbb{C}[\frac{y}{x}] \otimes_{\mathbb{C}[\frac{y}{x}]} \mathbb{C}[\frac{x}{y}] \to \frac{x}{y} \mathbb{C}[\frac{y}{x}] \otimes_{\mathbb{C}[\frac{y}{x}]} \mathbb{C}[\frac{x}{y}] \cong e_2^*\mathcal{O}([1:0])(U), $$
$$ \frac{x}{y} \otimes 1 \otimes 1 \mapsto -\frac{x}{y} \otimes 1, $$

$$ e_2^*\theta_{e_1}(V) : e_2^*e_1^*\mathcal{O}([1:0])(V) \cong \mathbb{C}[\frac{x}{y}] \otimes_{\mathbb{C}[\frac{x}{y}]} \mathbb{C}[\frac{x}{y}] \otimes_{\mathbb{C}[\frac{x}{y}]} \mathbb{C}[\frac{y}{x}] \to \mathbb{C}[\frac{x}{y}] \otimes_{\mathbb{C}[\frac{x}{y}]} \mathbb{C}[\frac{y}{x}] \cong e_2^*\mathcal{O}([1:0])(V), $$
$$ 1 \otimes 1 \otimes 1 \mapsto 1 \otimes 1, $$

$$ e_1^*\theta_{e_2}(U) : e_1^*e_2^*\mathcal{O}([1:0])(U) \cong \frac{x}{y} \mathbb{C}[\frac{y}{x}] \otimes_{\mathbb{C}[\frac{y}{x}]} \mathbb{C}[\frac{x}{y}] \otimes_{\mathbb{C}[\frac{x}{y}]} \mathbb{C}[\frac{x}{y}] \to \mathbb{C}[\frac{x}{y}] \otimes_{\mathbb{C}[\frac{x}{y}]} \mathbb{C}[\frac{x}{y}] \cong e_1^*\mathcal{O}([1:0])(U), $$
$$ \frac{x}{y} \otimes 1 \otimes 1 \mapsto 1 \otimes 1, $$

$$ e_1^*\theta_{e_2}(V) : e_1^*e_2^*\mathcal{O}([1:0])(V) \cong \mathbb{C}[\frac{x}{y}] \otimes_{\mathbb{C}[\frac{x}{y}]} \mathbb{C}[\frac{y}{x}] \otimes_{\mathbb{C}[\frac{y}{x}]} \mathbb{C}[\frac{y}{x}] \to \frac{x}{y} \mathbb{C}[\frac{y}{x}] \otimes_{\mathbb{C}[\frac{y}{x}]} \mathbb{C}[\frac{y}{x}] \cong e_1^*\mathcal{O}([1:0])(V),$$
$$ 1 \otimes 1 \otimes 1 \mapsto \frac{x}{y} \otimes 1. $$

Then we get $\alpha(e_1,e_2)=1$, $\alpha(e_2,e_1)=-1$.

Also to get $\alpha(e_1,e_1)$, $\alpha(e_2,e_2)$ from
$$ \theta_{0}=\alpha(e_1,e_1)\theta_{e_1} \circ {e_1}^*\theta_{e_1}, $$
$$ \theta_{0}=\alpha(e_2,e_2)\theta_{e_2} \circ {e_2}^*\theta_{e_2}, $$
we consider the following isomorphisms.

$$ e_1^*\theta_{e_1}(U) : e_1^*e_1^*\mathcal{O}([1:0])(U) \cong \mathbb{C}[\frac{x}{y}] \otimes_{\mathbb{C}[\frac{x}{y}]} \mathbb{C}[\frac{x}{y}] \otimes_{\mathbb{C}[\frac{x}{y}]} \mathbb{C}[\frac{x}{y}] \to \mathbb{C}[\frac{x}{y}] \otimes_{\mathbb{C}[\frac{x}{y}]} \mathbb{C}[\frac{x}{y}] \cong e_1^*\mathcal{O}([1:0])(U), $$
$$ 1 \otimes 1 \otimes 1 \mapsto 1 \otimes 1, $$

$$ e_1^*\theta_{e_1}(V) : e_1^*e_1^*\mathcal{O}([1:0])(V) \cong \frac{x}{y} \mathbb{C}[\frac{y}{x}] \otimes_{\mathbb{C}[\frac{y}{x}]} \mathbb{C}[\frac{y}{x}] \otimes_{\mathbb{C}[\frac{y}{x}]} \mathbb{C}[\frac{y}{x}] \to \frac{x}{y} \mathbb{C}[\frac{y}{x}] \otimes_{\mathbb{C}[\frac{y}{x}]} \mathbb{C}[\frac{y}{x}] \cong e_1^*\mathcal{O}([1:0])(V), $$
$$ \frac{x}{y} \otimes 1 \otimes 1 \mapsto -\frac{x}{y} \otimes 1, $$

$$ e_2^*\theta_{e_2}(U) : e_2^*e_2^*\mathcal{O}([1:0])(U) \cong \mathbb{C}[\frac{x}{y}] \otimes_{\mathbb{C}[\frac{x}{y}]} \mathbb{C}[\frac{y}{x}] \otimes_{\mathbb{C}[\frac{y}{x}]} \mathbb{C}[\frac{x}{y}] \to \frac{x}{y} \mathbb{C}[\frac{y}{x}] \otimes_{\mathbb{C}[\frac{y}{x}]} \mathbb{C}[\frac{x}{y}] \cong e_2^*\mathcal{O}([1:0])(U), $$
$$ 1 \otimes 1 \otimes 1 \mapsto \frac{x}{y} \otimes 1, $$

$$ e_2^*\theta_{e_2}(V) : e_2^*e_2^*\mathcal{O}([1:0])(V) \cong \frac{x}{y} \mathbb{C}[\frac{y}{x}] \otimes_{\mathbb{C}[\frac{y}{x}]} \mathbb{C}[\frac{x}{y}] \otimes_{\mathbb{C}[\frac{x}{y}]} \mathbb{C}[\frac{y}{x}] \to \mathbb{C}[\frac{x}{y}] \otimes_{\mathbb{C}[\frac{x}{y}]} \mathbb{C}[\frac{y}{x}] \cong e_2^*\mathcal{O}([1:0])(V),$$
$$ \frac{x}{y} \otimes 1 \otimes 1 \mapsto 1 \otimes 1. $$

Then we get $\alpha(e_1,e_1)=1$, $\alpha(e_2,e_2)=1$.

Let $N=(\mathbb{Z}/2)e_1$ and let $T=(\mathbb{Z}/2)e_2$ be a decomposition of $G = N \times T$. The above isomorphisms are selected so that our $\alpha$ to be normalized. Finally we can compute $\alpha$ as below using the general properties of the normal cycles stated in the previous subsection. \\

\begin{center}
\begin{tabular}{|c|c|c|c|c|} \hline
$\alpha$ & 0 & $e_1$ & $e_2$ & $e_1+e_2$ \\ \hline $0$ & 1 & 1 & 1 &
1 \\ \hline $e_1$ & 1 & 1 & 1 & 1 \\ \hline $e_2$ & 1 & -1 & 1 & -1
\\ \hline $e_1+e_2$ & 1 & -1 & 1 & -1 \\ \hline
\end{tabular}
\end{center}

\subsection{Pullback of invariant line bundle}

Let $\sigma$ be an involution of $C$ and suppose that $T=\langle
\sigma \rangle$ extends to a group of automorphisms $G=N \oplus T$
of $C$. Let $B=C/{\langle \sigma \rangle}$ be the quotient curve and
$\pi : C \to B$ be the quotient map. Let $L$ be an $N$-invariant
line bundles on $B$ whose Schur multiplier is $\beta \in
H^2(N,\mathbb{C}^*)$. In this case we can define $\theta_g:g^*\pi^*L
\to \pi^*L$ and compute the Schur multiplier of $\alpha \in
H^2(G,\mathbb{C})$ of $\pi^*L$ using $N$-invariant structure of $L$.

\begin{lem}
(1) $\pi^*L$ is a $G$-invariant line bundle on $C$. \\
(2) $ \alpha(n,n')=\beta(n,n'), \forall n,n' \in N.$ \\
(3) $ \alpha(\sigma,\sigma)=1 $, \\
(4) $ \alpha(\sigma,n)=\alpha(n,\sigma)=1, \forall n \in N$.
\end{lem}
\begin{proof}

We have the following three commutative diagram.

\begin{displaymath}
    \xymatrix{C \ar[rd]_{\pi} \ar[rr]^{\sigma} & & \ar[ld]^{\pi} C \\
               & B & }
\end{displaymath}

\begin{displaymath}
    \xymatrix{C \ar[d]^{\pi} \ar[r]^{n} & C \ar[d]^{\pi} \\
              B \ar[r]^{n} & B }
\end{displaymath}

\begin{displaymath}
    \xymatrix{C \ar[d]^{\pi} \ar[r]^{n \sigma = \sigma n} & C \ar[d]^{\pi} \\
              B \ar[r]^{n} & B }
\end{displaymath}

Then (1) is obvious. From the above diagrams we have isomorphisms $n^*\pi^*L
\cong \pi^*n^*L$, $\sigma^*\pi^*L \cong \pi^*\sigma^*L$, $(n \sigma)^*\pi^*L \cong
\pi^*n^*L$. These isomorphisms enable us to define
$\theta_n:n^*\pi^*L \to \pi^*L$, $\theta_{\sigma}:\sigma^*\pi^*L \to \pi^*L$,
$\theta_{n \sigma}:(n \sigma)^*\pi^*L \to \pi^*L$ on $C$ via the pullback of
$\theta_n : n^*L \to L$ on $B$. It suffices to check that these
isomorphisms satisfy (2), (3), (4) in a fixed affine chart. We leave
this to readers.
\end{proof}

\begin{rmk}
Let $G=N \oplus T$ be the decomposition of $G$ as above. Because
$\alpha$ is a normal cocyle we can compute $\alpha(g,g')$ for all
$g,g' \in G$ from the general properties of normal cycles.
\end{rmk}

\section{$G=(\mathbb{Z}/2)^3$ case}

Let $S=(C \times D)/G$ be a surface isogenous to a higher product with $p_g=q=0$ and $G=(\mathbb{Z}/2)^3$. In this case the curves are a hyperelliptic curve of genus 3 and an elliptic-hyperelliptic curve of genus 5. Let $C$ be the hyperelliptic curve of genus 3 and $D$ be the elliptic-hyperelliptic curve of genus 5. Let $\pi_C : C \to \mathbb{P}^1$, $\pi_D : D \to \mathbb{P}^1$ be the quotient maps. Then $\pi_C$ has 5 branch points and $\pi_D$ has 6 branch points. We may assume that the stabilizer elements of $C$ are $(e_1,e_2,e_3,e_1,e_2+e_3)$, and the stabilizer elements of $D$ are $(e_1+e_2,e_1+e_3,e_1+e_2+e_3,e_1+e_2,e_1+e_3,e_1+e_2+e_3)$. Let $E_1, E_2, E_3, E_4, E_5$ be the corresponding set-theoretic fibers on $C$ and let $F_1, F_2, F_3, F_4, F_5, F_6$ be the corresponding set-theoretic fibers on $D$. For detail, see \cite{BC}.

First we construct $G$-invariant theta characteristic $\kappa_C$ on $C$.

\begin{lem}
There exist $G$-invariant theta characteristic $\kappa_C$ on $C$ such that $h^0(\kappa_C) = h^1(\kappa_C) = 0$.
\end{lem}
\begin{proof}
$C$ is a hyperelliptic curve of genus 3. Let $\pi_C^1 : C \to
\mathbb{P}^1$ be the quotient by $(\mathbb{Z}/2)e_1$ action. We set
$\kappa_C={\pi_C^1}^*\mathcal{O}(-1) \otimes \mathcal{O}_C(E_1)$ be
a line bundle. From the Proposition of \cite{Be2} we see that
$\kappa_C$ is a theta characteristic with
$h^0(\kappa_C)=h^1(\kappa_C)=0$. Since $\mathcal{O}_C(E_1)$ is
$G$-equivariant line bundle, it suffice to show that
${\pi_C^1}^*\mathcal{O}(-1)$ is a $G$-invariant line bundle. The
action of $G$ induces $(\mathbb{Z}/2)e_2 \oplus
(\mathbb{Z}/2)e_3$-action on $\mathbb{P}^1$, and $\mathcal{O}(-1)$
is an invariant line bundle on $\mathbb{P}^1$. Therefore
${\pi_C^1}^*\mathcal{O}(-1)$ is a $G$-invariant line bundle on $C$.
\end{proof}

Next we construct $G$-invariant $2$-torsion line bundle $\eta_D$ on $D$ having the same Schur multiplier as that of  $\kappa_C$. Let $x,y,z \in G$ be the stabilizer elements of the $G$-action on
$D$ and let $F^x$ be a set-theoretic fiber which is a half of the
ramification points of the corresponding $(\mathbb{Z}/2)x$-quotient
of $D$. Then $G$-action induces an $(\mathbb{Z}/2)y \oplus
(\mathbb{Z}/2)z$-action on $F^x$ which is free. Then we can
decompose $F^x=F^x_1 \sqcup F^x_2$ such that one
$(\mathbb{Z}/2)y$-action preserves each $F^x_1, F^x_2$ and
$(\mathbb{Z}/2)z$-action exchanges them. Let us denote
$\eta^{y,z}_D=\mathcal{O}(F^x_1 - F^x_2)$. Note that $\eta^{y,z}$ is not trivial by Lemma 3.1.

\begin{lem}
$\eta^{y,z}_D$ is a $G$-invariant $2$-torsion line bundle on $D$.
\end{lem}
\begin{proof}
Consider the $(\mathbb{Z}/2)x \oplus (\mathbb{Z}/2)y$-action on $D$, and let $\pi^{x,y}_D: D \to \mathbb{P}^1$ be the quotient morphism. $2F^x_1 \sim 2F^x_2$ follows since $2F^x_1 \sim 2F^x_2 \sim (\pi^{x,y}_D)^*{pt}$. Therefore $\eta^{y,z}_D$ is a $2$-torsion line bundle. Since $(\mathbb{Z}/2)x \oplus (\mathbb{Z}/2)y$-action fixes $\eta^{y,z}_D$ and $z^*(\eta^{y,z}_D) \cong (\eta^{y,z}_D)^{\otimes{-1}} \cong \eta^{y,z}_D$, $\eta^{y,z}_D$ is a $G$-invariant line bundle on $D$.
\end{proof}

Therefore we get three $2$-torsion line bundles $\eta^{e_1+e_3,e_1+e_2+e_3}_D,\eta^{e_1+e_2+e_3,e_1+e_2}_D,\eta^{e_1+e_3,e_1+e_2}_D$ on $D$, and let  $\eta_D=\eta^{e_1+e_3,e_1+e_2+e_3}_D \otimes \eta^{e_1+e_2+e_3,e_1+e_2}_D \otimes \eta^{e_1+e_3,e_1+e_2}_D$.

\begin{prop}
Let $\kappa_C$ be the $G$-invariant theta characteristics
constructed above, and let $\eta_C$ be the $G$-invariant torsion
line bundles constructed above. Then $
\alpha(\kappa_C)\alpha(\eta_D) = 0 \in H^2(G,\mathbb{C}^*). $
\end{prop}
\begin{proof}
Using the above basic example we can compute the cocycles of
$\alpha(\kappa_C), \alpha(\eta_D) \in H^2(G,\mathbb{C}^*).$ Since
$\mathcal{O}_C(E_1)$ is a $G$-equivariant line bundle, it suffice to
compute $\alpha((\pi_C^1)^*\mathcal{O}(-1)).$ From the calculation of basic example we get the following table for
$\alpha((\pi_C^1)^*\mathcal{O}(-1)).$
\begin{center}
\begin{tabular}{|c|c|c|c|c|c|c|c|c|} \hline
$\alpha$ & 0 & $e_1$ & $e_2$ & $e_1+e_2$ & $e_3$ & $e_1+e_3$ & $e_2+e_3$ & $e_1+e_2+e_3$\\
\hline
$0$ & 1 & 1 & 1 & 1 & 1 & 1 & 1 & 1 \\
\hline
$e_1$ & 1 & 1 & 1 & 1 & 1 & 1 & 1 & 1 \\
\hline
$e_2$ & 1 & 1 & 1 & 1 & 1 & 1 & 1 & 1 \\
\hline
$e_1+e_2$ & 1 & 1 & 1 & 1 & 1 & 1 & 1 & 1 \\
\hline
$e_3$ & 1 & 1 & -1 & -1 & 1 & 1 & -1 & -1 \\
\hline
$e_1+e_3$ & 1 & 1 & -1 & -1 & 1 & 1 & -1 & -1 \\
\hline
$e_2+e_3$ & 1 & 1 & -1 & -1 & 1 & 1 & -1 & -1 \\
\hline
$e_1+e_2+e_3$ & 1 & 1 & -1 & -1 & 1 & 1 & -1 & -1 \\
\hline
\end{tabular}
\end{center}

We can also compute $\alpha(\eta^{e_1+e_3,e_1+e_2+e_3}_D)$ as follows.

\begin{center}
\begin{tabular}{|c|c|c|c|c|c|c|c|c|} \hline
$\alpha$ & 0 & $e_1$ & $e_2$ & $e_1+e_2$ & $e_3$ & $e_1+e_3$ & $e_2+e_3$ & $e_1+e_2+e_3$\\
\hline
$0$ & 1 & 1 & 1 & 1 & 1 & 1 & 1 & 1 \\
\hline
$e_1$ & 1 & -1 & -1 & 1 & 1 & -1 & -1 & 1 \\
\hline
$e_2$ & 1 & -1 & -1 & 1 & 1 & -1 & -1 & 1 \\
\hline
$e_1+e_2$ & 1 & 1 & 1 & 1 & 1 & 1 & 1 & 1 \\
\hline
$e_3$ & 1 & -1 & -1 & 1 & 1 & -1 & -1 & 1 \\
\hline
$e_1+e_3$ & 1 & 1 & 1 & 1 & 1 & 1 & 1 & 1 \\
\hline
$e_2+e_3$ & 1 & 1 & 1 & 1 & 1 & 1 & 1 & 1 \\
\hline
$e_1+e_2+e_3$ & 1 & -1 & -1 & 1 & 1 & -1 & -1 & 1 \\
\hline
\end{tabular}
\end{center}

We can also compute $\alpha(\eta^{e_1+e_2+e_3,e_1+e_2}_D)$ as follows.

\begin{center}
\begin{tabular}{|c|c|c|c|c|c|c|c|c|} \hline
$\alpha$ & 0 & $e_1$ & $e_2$ & $e_1+e_2$ & $e_3$ & $e_1+e_3$ & $e_2+e_3$ & $e_1+e_2+e_3$\\
\hline
$0$ & 1 & 1 & 1 & 1 & 1 & 1 & 1 & 1 \\
\hline
$e_1$ & 1 & -1 & -1 & 1 & -1 & 1 & 1 & -1 \\
\hline
$e_2$ & 1 & 1 & 1 & 1 & 1 & 1 & 1 & 1 \\
\hline
$e_1+e_2$ & 1 & -1 & -1 & 1 & -1 & 1 & 1 & -1 \\
\hline
$e_3$ & 1 & -1 & -1 & 1 & -1 & 1 & 1 & -1 \\
\hline
$e_1+e_3$ & 1 & 1 & 1 & 1 & 1 & 1 & 1 & 1 \\
\hline
$e_2+e_3$ & 1 & -1 & -1 & 1 & -1 & 1 & 1 & -1 \\
\hline
$e_1+e_2+e_3$ & 1 & 1 & 1 & 1 & 1 & 1 & 1 & 1 \\
\hline
\end{tabular}
\end{center}

We can also compute $\alpha(\eta^{e_1+e_3,e_1+e_2}_D)$ as follows.

\begin{center}
\begin{tabular}{|c|c|c|c|c|c|c|c|c|} \hline
$\alpha$ & 0 & $e_1$ & $e_2$ & $e_1+e_2$ & $e_3$ & $e_1+e_3$ & $e_2+e_3$ & $e_1+e_2+e_3$\\
\hline
$0$ & 1 & 1 & 1 & 1 & 1 & 1 & 1 & 1 \\
\hline
$e_1$ & 1 & -1 & -1 & 1 & 1 & -1 & -1 & 1 \\
\hline
$e_2$ & 1 & 1 & 1 & 1 & 1 & 1 & 1 & 1 \\
\hline
$e_1+e_2$ & 1 & -1 & -1 & 1 & 1 & -1 & -1 & 1 \\
\hline
$e_3$ & 1 & -1 & -1 & 1 & 1 & -1 & -1 & 1 \\
\hline
$e_1+e_3$ & 1 & 1 & 1 & 1 & 1 & 1 & 1 & 1 \\
\hline
$e_2+e_3$ & 1 & -1 & -1 & 1 & 1 & -1 & -1 & 1 \\
\hline
$e_1+e_2+e_3$ & 1 & 1 & 1 & 1 & 1 & 1 & 1 & 1 \\
\hline
\end{tabular}
\end{center}

Then $\alpha(\kappa_C) \alpha(\eta_D)$ becomes as follows.

\begin{center}
\begin{tabular}{|c|c|c|c|c|c|c|c|c|} \hline
$\alpha$ & 0 & $e_1$ & $e_2$ & $e_1+e_2$ & $e_3$ & $e_1+e_3$ & $e_2+e_3$ & $e_1+e_2+e_3$\\
\hline
$0$ & 1 & 1 & 1 & 1 & 1 & 1 & 1 & 1 \\
\hline
$e_1$ & 1 & -1 & -1 & 1 & -1 & 1 & 1 & -1 \\
\hline
$e_2$ & 1 & -1 & -1 & 1 & 1 & -1 & -1 & 1 \\
\hline
$e_1+e_2$ & 1 & 1 & 1 & 1 & -1 & -1 & -1 & -1 \\
\hline
$e_3$ & 1 & -1 & 1 & -1 & -1 & 1 & -1 & 1 \\
\hline
$e_1+e_3$ & 1 & 1 & -1 & -1 & 1 & 1 & -1 & -1 \\
\hline
$e_2+e_3$ & 1 & 1 & -1 & -1 & -1 & -1 & 1 & 1 \\
\hline
$e_1+e_2+e_3$ & 1 & -1 & 1 & -1 & 1 & -1 & 1 & -1 \\
\hline
\end{tabular}
\end{center}

Finally we can show that the above cocycle becomes a coboundary by
giving
$t(0)=1,t(e_1)=\sqrt{-1},t(e_2)=\sqrt{-1},t(e_3)=\sqrt{-1},t(e_1+e_2)=1,t(e_1+e_3)=1,t(e_2+e_3)=-1,t(e_1+e_2+e_3)=-\sqrt{-1},$
and check that $\alpha(x,y)=t(x)t(y)t(xy)^{-1}, \forall x,y\in G$.
\end{proof}

\begin{rmk}
It follows that $ \kappa_C \boxtimes \eta_D $ is a $G$-equivariant line bundle on $ C \times D $.
\end{rmk}

\begin{lem} The 6 effective $G$ invariant divisors of degree 4 divide into 3 divisor classes $F_1 \sim F_4$, $F_2 \sim F_5$, $F_3 \sim F_6$ which are not linearly equivalent to each other.
\end{lem}
\begin{proof}
Consider $(\mathbb{Z}/2)(e_1+e_3) \oplus (\mathbb{Z}/2)(e_1+e_2+e_3)$-action on $D$ and let $\pi^{e_1+e_3,e_1+e_2+e_3}:D \to \mathbb{P}^1$ be the quotient map. $F_1 \sim F_4$ follows since they are pullbacks of the point of
$\mathbb{P}^1$ via $\pi^{e_1+e_3,e_1+e_2+e_3}$. Similarly we get $F_2 \sim F_5$, $F_3 \sim F_6$. Consider $(\mathbb{Z}/2)(e_1+e_3)$-action on $D$ and let $\pi^{e_1+e_3}:D \to E$ be the quotient map. Then $F_1$ is a pullback of a line bundle on $E$, and $F_2$ is half of the ramification point of $\pi^{e_1+e_3}$. From the Lemma 3.1 we get $F_1 \nsim F_2$. Similarly we see that $F_1
\nsim F_3$.

\end{proof}

From the above two lemmas, we find that every $G$-invariant effective divisor on $E$ of degree 4 is linearly equivalent to $F_1$ or $F_2$ or $F_3$. With the same notation, we have the following lemma.

\begin{lem}
$$ h^0(D, \mathcal{O}_D(F_1+F_2-F_3)) = 0, $$
$$ h^1(D, \mathcal{O}_D(F_1+F_2-F_3)) = 0, $$
$$ h^0(D, \mathcal{O}_D(-F_1-F_2+F_3)) = 0, $$
$$ h^1(D, \mathcal{O}_D(-F_1-F_2+F_3)) = 8. $$
\end{lem}
\begin{proof}
From the Riemann-Roch formula we find that $$ h^0(D, \mathcal{O}_D(F_1+F_2-F_3)) - h^1(D, \mathcal{O}_
D(F_1+F_2-F_3)) = 1+4-5 = 0. $$
Therefore it suffices to show that $h^0(D, \mathcal{O}_D(F_1+F_2-F_3)) = 0.$ We know that $F_1, F_2, F_3$ are $G$-invariant divisors on $D$ and hence there is a $G$-action on $H^0(D,\mathcal{O}_D(F_1+F_2-F_3))$. If $ h^0(D, \mathcal{O}_D(F_1+F_2-F_3)) \neq 0, $ then there is a $G$-eigensection $f \in H^0(D, \mathcal{O}_D(F_1+F_2-F_3))$, and $F_1+F_2-F_3+(f)$ should be a $G$-invariant effective divisor of degree 4. Every $G$-invariant effective divisor of degree 4 on $D$ is linearly equivalent to $F_1$ or $F_2$ or $F_3$ by the above lemma. It follows that $F_1+F_2-F_3 \sim F_1$ or $F_1+F_2-F_3 \sim F_2$ or $F_1+F_2-F_3 \sim F_3$. Then $ F_2-F_3 \sim 0 $ or $ F_1-F_3 \sim 0 $ or $ F_1 + F_2 \sim 2F_3 \sim 2F_1$ which contradicts the assumption that $F_1$, $F_2$ and $F_3$ are not linearly equivalent to each other.

From the Riemann-Roch theorem we get $$ h^0(D,
\mathcal{O}_D(-F_1-F_2+F_3)) - h^1(D, \mathcal{O}_D(-F_1-F_2+F_3)) =
1-4-5 = -8. $$ and $$ h^0(D, \mathcal{O}_D(-F_1-F_2+F_3)) = 0 $$
because the degree of $\mathcal{O}_D(-F_1-F_2+F_3)$ is negative.
\end{proof}

\begin{rmk}
Because the $G$-action on $C \times D$ is free we have $D^b(S)
\simeq D^b_G(C \times D)$ and every $G$-equivariant line bundle on
$C \times D$ corresponds to a line bundle on $S$. Therefore we
regards a $G$-equivariant line bundle on $C \times D$ as a line
bundle on $S$
\end{rmk}

\begin{thm} Let $S=(C \times D)/G$ be a surface isogenous to a higher product with $p_g=q=0$, $G=(\mathbb{Z}/2)^3$. For any choice of four characters $\chi_1, \chi_2, \chi_3, \chi_4$, the following sequence
$$ \mathcal{O}_{C \times D}(\chi_1), q^*\mathcal{O}_D(-F_1-F_2+F_3)(\chi_2), \kappa_C^{-1} \boxtimes \eta_D(\chi_3), \kappa_C^{-1} \boxtimes (\eta_D \otimes \mathcal{O}_D(-F_1-F_2+F_3))(\chi_4) $$ is a exceptional sequence of line bundles of maximal length in $D^b(S)$.
\end{thm}
\begin{proof}
Since $p_g=q=0$, every line bundle on $S$ is exceptional.
From the K\"{u}nneth formula we find that
$$ h^j(C \times D, q^*\mathcal{O}_D(F_1+F_2-F_3)) = 0, \forall j, $$
$$ h^j(C \times D, \kappa_C \boxtimes \eta_D) = 0, \forall j, $$
$$ h^j(C \times D, \kappa_C \boxtimes (\eta_D \otimes \mathcal{O}_D(F_1+F_2-F_3))) = 0, \forall j, $$
$$ h^j(C \times D, \kappa_C \boxtimes (\eta_D \otimes \mathcal{O}_D(-F_1-F_2+F_3))) = 0, \forall j. $$
Therefore the $G$-invariant parts are also trivial. Hence, we find
that $ \mathcal{O}_{C \times D}(\chi_1),
q^*\mathcal{O}_D(-F_1-F_2+F_3)(\chi_2), \kappa_C^{-1} \boxtimes
\eta_D(\chi_3), \kappa_C^{-1} \boxtimes (\eta_D \otimes
\mathcal{O}_D(-F_1-F_2+F_3))(\chi_4) $ form an exceptional sequence.
Since the rank of $K(S)$ is 4, the maximal length of exceptional
sequences on $S$ is 4.
\end{proof}

\begin{prop}
Let $\mathcal{A}$ be the orthogonal complement of an exceptional
sequence $ \mathcal{O}_{C \times D}(\chi_1)$,
$q^*\mathcal{O}_D(-F_1-F_2+F_3)(\chi_2)$, $\kappa_C^{-1} \boxtimes
\eta_D(\chi_3)$, $\kappa_C^{-1} \boxtimes (\eta_D \otimes
\mathcal{O}_D(-F_1-F_2+F_3))(\chi_4) $. Then $\mathcal{A}$ is a
quasiphantom category whose Grothendieck group is isomorphic to
$(\mathbb{Z}/2)^4 \oplus (\mathbb{Z}/4)^2$.
\end{prop}
\begin{proof}
Since the Betti number of $S$ is 4, we see that the orthogonal
complement of an exceptional sequence is a quasiphantom category
from Kuznetsov's theorem \cite{Ku3}.
\end{proof}

Then we can compute the Hochschild cohomologies of the quasiphantom categories. See \cite{Ku4} for the definitions and more details.

\begin{prop}
The pseudoheight of the exceptional sequence $ \mathcal{O}_{C \times
D}(\chi_1)$, $q^*\mathcal{O}_D(-F_1-F_2+F_3)(\chi_2)$,
$\kappa_C^{-1} \boxtimes \eta_D(\chi_3)$, $\kappa_C^{-1} \boxtimes
(\eta_D \otimes \mathcal{O}_D(-F_1-F_2+F_3))(\chi_4) $ is 4 and the
height is 4.
\end{prop}
\begin{proof}
From the K\"{u}nneth formula and degree computation we find that
$\mathcal{O}_{C \times D}(\chi_1)$,
$q^*\mathcal{O}_D(-F_1-F_2+F_3)(\chi_2)$, $\kappa_C^{-1} \boxtimes
\eta_D(\chi_3)$, $\kappa_C^{-1} \boxtimes (\eta_D \otimes
\mathcal{O}_D(-F_1-F_2+F_3))(\chi_4), \mathcal{O}_{C \times
D}(\chi_1) \otimes \omega_S^{-1}$, $q^*\mathcal{O}_D(-F_1-F_2+F_3)(\chi_2) \otimes \omega_S^{-1}$,
$\kappa_C^{-1} \boxtimes \eta_D(\chi_3) \otimes \omega_S^{-1}$, $\kappa_C^{-1} \boxtimes
(\eta_D \otimes \mathcal{O}_D(-F_1-F_2+F_3))(\chi_4) \otimes \omega_S^{-1}$ is Hom-free. This sequence
cannot be cyclically $Ext^1$-connected by Serre duality and Kodaira
vanishing theorem.
\end{proof}

Therefore we get the following consequence about the Hochschild
cohomologies of our quasiphantom categories.

\begin{cor}
Let $\mathcal{A}$ be the orthogonal complement of the exceptional
collection $ \mathcal{O}_{C \times D}(\chi_1)$,
$q^*\mathcal{O}_D(-F_1-F_2+F_3)(\chi_2)$, $\kappa_C^{-1} \boxtimes
\eta_D(\chi_3)$, $\kappa_C^{-1} \boxtimes (\eta_D \otimes
\mathcal{O}_D(-F_1-F_2+F_3))(\chi_4) $. Then we have $ HH^i(S) =
HH^i(\mathcal{A})$, for $i=0,1,2 $, and $ HH^3(S) \subset
HH^3(\mathcal{A}) $.
\end{cor}

\section{$G=(\mathbb{Z}/2)^4$ case}

Let $S=(C \times D)/G$ be a surface isogenous to a higher product with $p_g=q=0$ and $G=(\mathbb{Z}/2)^4$. In this case the curves are elliptic-hypereilliptic curves of genus 5. Let $C$ and $D$ be the elliptic-hypereilliptic curves of genus 5. Let $\pi_C : C \to \mathbb{P}^1$, $\pi_D : D \to \mathbb{P}^1$ be the quotient maps. Then $\pi_C$ and $\pi_D$ have 5 branch points. We may assume that the stabilizer elements of $C$ are $(e_1,e_2,e_3,e_4,e:=e_1+e_2+e_3+e_4)$, and the stabilizer elements of $D$ are $(e+e_1,e+e_2,e_1+e_3,e_2+e_4,e_3+e_4)$. Let $E_1, E_2, E_3, E_4, E_5$ be the corresponding set-theoretic fibers on $C$ and let $F_1, F_2, F_3, F_4, F_5$ be the corresponding set-theoretic fibers on $D$. For detail, see \cite{BC}. \\

Now we construct $G$-invariant theta charactristics on $C$ and $D$.
Let $x,y,z,w,x+y+z+w \in G$ be the stabilizer elements. Consider
$(\mathbb{Z}/2)x \oplus (\mathbb{Z}/2)y$-action on $C$ and let
$\pi^{x,y} : C \to \mathbb{P}^1$ be the $(\mathbb{Z}/2)x \oplus
(\mathbb{Z}/2)y$-quotient of $C$. Then $G$-action induces an
$(\mathbb{Z}/2)z \oplus (\mathbb{Z}/2)w$-action on $\mathbb{P}^1$.
Note that the set theoretic fiber of $\pi_C$ with stabilizer element
$z$ is the union of pullbacks of two points $p,q$ on $\mathbb{P}^1$
by $\pi^{x,y}$. Take one such point $p$ on $\mathbb{P}^1$ and denote $E^{x,y,z,w}=(\pi^{x,y})^*p$. In this way we can define three $G$-invariant line bundles $\mathcal{O}_C(E^{e_2,e_4,e_3,e_1})$, $\mathcal{O}_C(E^{e_1,e_2,e,e_3})$,
$\mathcal{O}_C(E^{e_2,e,e_4,e_1})$ on $C$. Let us denote
$\kappa_C=\mathcal{O}_C(E^{e_2,e_4,e_3,e_1})$, and
$\eta_C=\mathcal{O}_C(E^{e_1,e_2,e,e_3}-E^{e_2,e,e_4,e_1})$.

\begin{lem}
(1) $E^{x,y,z,w}$ is a $G$-invariant theta characteristic on $C$. \\
(2) $\eta_C$ is a $G$-invariant torsion line bundle on $C$. \\
(3) $\kappa_C \otimes \eta_C$ is a $G$-invariant theta
characteristic with $h^0(C,\kappa_C \otimes \eta_C)=h^1(C,\kappa_C
\otimes \eta_C)=2$.
\end{lem}
\begin{proof}
(1) follows immediately since $(\pi^{x,y})^*p+(\pi^{x,y})^*q$ is a
canonical divisor on $C$ and $(\pi^{x,y})^*p \sim (\pi^{x,y})^*q$.
(2) is obvious. Consider $e_2$-action on $C$. Then $\kappa_C \otimes \eta_C$ is a pullback of degree 2 line bundles on $C/{\langle e_2 \rangle}$. From Proposition 3.3 and Riemann-Roch formula we have (3).
\end{proof}

Similarly we can define three theta characteristics $\mathcal{O}_D(F^{e+e_1,e_3+e_4,e_2+e_4,e_1+e_3})$, \\ $\mathcal{O}_D(F^{e_1+e_3,e_2+e_4,e+e_1,e+e_2})$,$\mathcal{O}_D(F^{e+e_2,e_2+e_4,e_1+e_3,e_3+e_4})$ on $D$ in a
similar way. Let us denote $\kappa_D=\mathcal{O}_D(F^{e+e_1,e_3+e_4,e_2+e_4,e_1+e_3})$, $\eta_D=\mathcal{O}_D(F^{e_1+e_3,e_2+e_4,e+e_1,e+e_2}-F^{e+e_2,e_2+e_4,e_1+e_3,e_3+e_4})$

\begin{lem}
(1) $F^{x,y,z,w}$ is a $G$-invariant theta characteristic on $D$. \\
(2) $\eta_D$ is a $G$-invariant torsion line bundle on $D$. \\
(3) $\kappa_D \otimes \eta_D$ is a $G$-invariant theta
characteristic with $h^0(D,\kappa_D \otimes \eta_D)=h^1(D,\kappa_D
\otimes \eta_D)=0$.
\end{lem}
\begin{proof}
The only nontrivial part is (3). Consider $e_2+e_4$-action on $D$,
and let $\pi^{e_2+e_4}:D \to E$ be the quotient map. Then $\kappa_D
\otimes \eta_D \cong (\pi^{e_2+e_4})^*(L) \otimes
\mathcal{O}_D(F^{e+e_1,e_3+e_4,e_2+e_4,e_1+e_3})$ for some $L$. From
the Proposition 3.3 we get $h^0(D,\kappa_D \otimes
\eta_D)=h^0(L)+h^1(L)$. Consider $e_1+e_3$-action of $E$. Then we
can prove that $L$ is a noneffective theta characteristic on $E$ by the same argument. Therefore we get $h^0(D,\kappa_D
\otimes \eta_D)=h^1(D,\kappa_D \otimes \eta_D)=0$.
\end{proof}

\begin{prop}
Let $\kappa_C, \kappa_D$ be the $G$-invariant theta characteristics
constructed above, and let $\eta_C, \eta_D$ be the $G$-invariant
$2$-torsion line bundles constructed above. Then $
\alpha(\kappa_C)\alpha(\eta_D) = \alpha(\eta_C)\alpha(\kappa_D) = 0
\in H^2(G,\mathbb{C}^*). $
\end{prop}
\begin{proof}
As the $(\mathbb{Z}/2)^3$-case we can compute $\alpha(\kappa_C)$,
$\alpha(\eta_C)$, $\alpha(\kappa_D)$, $\alpha(\eta_D)$ (see the
table in the last part of the paper). Finally we can show that
$\alpha(\kappa_C)\alpha(\eta_D)$ is a coboundary by giving
$t(0)=1,t(e_1)=1,t(e_2)=1,t(e_1+e_2)=-1,t(e_3)=\sqrt{-1},t(e_1+e_3)=\sqrt{-1},t(e_2+e_3)=-\sqrt{-1},t(e_1+e_2+e_3)=\sqrt{-1},t(e_4)=1,t(e_1+e_4)=-1,t(e_2+e_4)=1,,t(e_1+e_2+e_4)=1,t(e_3+e_4)=-\sqrt{-1},t(e_1+e_3+e_4)=\sqrt{-1},t(e_2+e_3+e_4)=\sqrt{-1},t(e_1+e_2+e_3+e_4)=\sqrt{-1}.$
and check that $\alpha(x,y)=t(x)t(y)t(xy)^{-1}, \forall x,y\in G$.
\\
Similarly we can show that $\alpha(\eta_C)\alpha(\kappa_D)$ is a
coboundary by
giving$t(0)=1,t(e_1)=\sqrt{-1},t(e_2)=1,t(e_1+e_2)=\sqrt{-1},t(e_3)=\sqrt{-1},t(e_1+e_3)=-1,t(e_2+e_3)=\sqrt{-1},t(e_1+e_2+e_3)=-1,t(e_4)=\sqrt{-1},t(e_1+e_4)=-1,t(e_2+e_4)=\sqrt{-1},t(e_1+e_2+e_4)=-1,t(e_3+e_4)=1,t(e_1+e_3+e_4)=-\sqrt{-1},t(e_2+e_3+e_4)=1,t(e_1+e_2+e_3+e_4)=-\sqrt{-1}.$
and check that $\alpha(x,y)=t(x)t(y)t(xy)^{-1}, \forall x,y\in G$.
\\
Therefore we have $\alpha(\kappa_C)\alpha(\eta_D) =
\alpha(\eta_C)\alpha(\kappa_D) = 0 \in H^2(G,\mathbb{C}^*) $
\end{proof}

\begin{rmk}
(1) Because the $G$-action on $C \times D$ is free we have $D^b(S)
\simeq D^b_G(C \times D)$ and every $G$-equivariant line bundle on
$C \times D$ corresponds to a line bundle on $S$. Therefore we
regards a $G$-equivariant line bundle on $C \times D$ as a line
bundle on $S$. \\
(2) The above proposition implies that we can regard $\kappa_C
\boxtimes \eta_D$, $\eta_C \boxtimes \kappa_D$ as line bundles on
$S$.
\end{rmk}

\begin{thm} Let $S=(C \times D)/G$ be a surface isogenous to a higher product with $p_g=q=0$, $G=(\mathbb{Z}/2)^4$. Then there exist two characters $\chi_1, \chi_2 \in \widehat{G}$ such that the following sequence
$$ \mathcal{O}, \kappa_C^{-1} \boxtimes \eta_D(\chi_1), \eta_C \boxtimes \kappa_D^{-1}(\chi_2), (\kappa_C^{-1} \otimes \eta_C) \boxtimes (\kappa_D^{-1} \otimes \eta_D)(\chi_1+\chi_2) $$
is a exceptional sequence of maximal length on $D^b(S)$.
\end{thm}
\begin{proof}
Since $p_g=q=0$, every line bundle on $S$ is exceptional. For
simplicity let $\mathcal{O}$, $L_1^{-1}$, $L_2^{-1}$, $(L_1 \otimes
L_2)^{-1}$ be the line bundles on $S$ corresponding to $
\mathcal{O}, \kappa_C^{-1} \boxtimes \eta_D(\chi_1), \eta_C
\boxtimes \kappa_D^{-1}(\chi_2), (\kappa_C^{-1} \otimes \eta_C)
\boxtimes (\kappa_D^{-1} \otimes \eta_D)(\chi_1+\chi_2) $ on $C
\times D$. We have to show that $H^i(S,L_1)=0$, $H^i(S,L_2)=0$,
$H^i(S,L_1^{-1} \otimes L_2)=0$, $H^i(S,L_1 \otimes L_2)=0$ for all $i$. Note that Riemann-Roch formula implies that $\chi(S,L_1)=0$, $\chi(S,L_2)=0$, $\chi(S,L_1^{-1} \otimes L_2)=0$, $\chi(S,L_1 \otimes L_2)=0$.

From Proposition 3.3 we see that $h^0(C,\kappa_C)=h^1(C,\kappa_C)=2$, $h^0(D,\kappa_D)=h^1(D,\kappa_D)=2$, and from Riemann-Roch we get $h^0(C,\eta_C)=0$, $h^1(C,\eta_C)=4$, $h^0(D,\eta_D)=0$, $h^1(D,\eta_D)=4$. From this we get $h^2(C \times D, \kappa_C \boxtimes \eta_D)=8$ and $h^2(C \times D, \eta_C \boxtimes \kappa_D)=8$. Therefore there exist $\chi_1, \chi_2 \in \widehat{G}$ such that $H^2(C \times D, \kappa_C \boxtimes \eta_D(\chi_1))^G=0$, $H^2(C \times D, \eta_C \boxtimes \kappa_D(\chi_2))^G=0$. Let $L_1^{-1}=\kappa_C^{-1} \boxtimes \eta_D(\chi_1)$, $L_2^{-1}=\eta_C
\boxtimes \kappa_D^{-1}(\chi_2)$. Then we have $H^0(S,L_1)=H^2(S,L_1)=0$ and $H^1(S,L_1)=0$ follows since $\chi(L_1)=0$. Similarly we get $H^i(S,L_2)=0$, for all $i$.

It remains to show that $H^i(S,L_1 \otimes L_2)=H^i(S,L_1^{-1} \otimes L_2)=0$, for all
$i$. Recall that $L_1^{-1} \otimes L_2$  corresponds to
$(\kappa_C^{-1} \otimes \eta_C) \boxtimes (\kappa_D \otimes
\eta_D)(\chi_1+\chi_2)$ on $C \times D$. From the K\"{u}nneth
formula we find that $H^0(C \times D, (\kappa_C^{-1} \otimes \eta_C)
\boxtimes (\kappa_D \otimes \eta_D))=0$ for degree reason. Finally
$H^1(C \times D, (\kappa_C^{-1} \otimes \eta_C) \boxtimes (\kappa_D
\otimes \eta_D))=H^0(C, \kappa_C^{-1} \otimes \eta_C) \otimes
H^1(D,\kappa_D \otimes \eta_D) \oplus H^1(C, \kappa_C^{-1} \otimes
\eta_C) \otimes H^0(D,\kappa_D \otimes \eta_D)=0$, $H^2(C \times D,
(\kappa_C^{-1} \otimes \eta_C) \boxtimes (\kappa_D \otimes
\eta_D))=H^1(C, \kappa_C^{-1} \otimes \eta_C) \otimes H^1(D,\kappa_D
\otimes \eta_D)=0$ since $\kappa_D \otimes \eta_D$ is a noneffective
theta characteristic on $D$. We can prove $H^i(S, L_1 \otimes L_2)=0$, for all
$i$ similarly. Therefore we get the desired result.
Since the rank of $K(S)$ is 4, the maximal length of exceptional
sequences on $S$ is 4.
\end{proof}

\begin{prop}
Let $\mathcal{A}$ be the orthogonal complement of an exceptional
sequence $ \mathcal{O}, \kappa_C^{-1} \boxtimes \eta_D(\chi_1),
\eta_C \boxtimes \kappa_D^{-1}(\chi_2), (\kappa_C^{-1} \otimes
\eta_C) \boxtimes (\kappa_D^{-1} \otimes \eta_D)(\chi_1+\chi_2) $.
Then $\mathcal{A}$ is a quasiphantom category whose Grothendieck
group is isomorphic to $(\mathbb{Z}/4)^4$.

\end{prop}
\begin{proof}
Since the Betti number of $S$ is 4, we see that the orthogonal
complement of an exceptional sequence is a quasiphantom category
from Kuznetsov's theorem \cite{Ku3}.
\end{proof}

We also prove that the DG algebra of endomorphisms of the
exceptional sequences constructed above are deformation invariant.

\begin{prop}
The DG algebra of endomorphisms of $T = \mathcal{O} \oplus
\kappa_C^{-1} \boxtimes \eta_D(\chi_1) \oplus \eta_C \boxtimes
\kappa_D^{-1}(\chi_2) \oplus (\kappa_C^{-1} \otimes \eta_C)
\boxtimes (\kappa_D^{-1} \otimes \eta_D)(\chi_1+\chi_2) $ does not
change under small deformations of the complex structure of $S$.
\end{prop}
\begin{proof}
From the K\"{u}nneth formula we get the followings.
$$ H^1(C \times D,\kappa_C^{-1} \boxtimes \eta_D)=0, $$
$$ H^1(C \times D,\eta_C \boxtimes \kappa_D^{-1})=0, $$
$$ H^1(C \times D,(\kappa_C^{-1} \otimes \eta_C) \boxtimes (\kappa_D^{-1} \otimes \eta_D))=0. $$
$$ H^1(C \times D,(\kappa_C \otimes \eta_C) \boxtimes (\kappa_D^{-1} \otimes \eta_D))=\mathbb{C}^{16}. $$
Consider the minimal model of the DG algebra of endomorphism of $T =
\mathcal{O} \oplus \kappa_C^{-1} \boxtimes \eta_D(\chi_1) \oplus
\eta_C \boxtimes \kappa_D^{-1}(\chi_2) \oplus (\kappa_C^{-1} \otimes
\eta_C) \boxtimes (\kappa_D^{-1} \otimes \eta_D)(\chi_1+\chi_2) $. 
We see that the minimal model of the DG algebra
$RHom^*(T,T)$ is formal by degree reason. We also see that $m_2$ is also 0 since there is no $Ext^3$ or $Ext^4$ between objects. By semicontinuity we see that the dimension of $H^*(RHom^*(T,T))$ is constant and the algebra structure of the minimal algebra $H^*(RHom^*(T,T))$ does not change by small deformations of the complex structure of $S$.
\end{proof}

\begin{prop}
The pseudoheight of the exceptional collection $ \mathcal{O}$,
$\kappa_C^{-1} \boxtimes \eta_D(\chi_1)$, $\eta_C \boxtimes
\kappa_D^{-1}(\chi_2)$, $(\kappa_C^{-1} \otimes \eta_C) \boxtimes
(\kappa_D^{-1} \otimes \eta_D)(\chi_1+\chi_2) $ is 4 and the height
is 4.
\end{prop}
\begin{proof}
From the K\"{u}nneth formula and degree computation we find that $ \mathcal{O}, \kappa_C^{-1} \boxtimes \eta_D(\chi_1), \eta_C \boxtimes \kappa_D^{-1}(\chi_2), (\kappa_C^{-1} \otimes \eta_C) \boxtimes (\kappa_D^{-1} \otimes \eta_D)(\chi_1+\chi_2), \mathcal{O} \otimes \omega_S^{-1}, \kappa_C^{-1} \boxtimes \eta_D(\chi_1) \otimes \omega_S^{-1}, \eta_C \boxtimes \kappa_D^{-1}(\chi_2) \otimes \omega_S^{-1}, (\kappa_C^{-1} \otimes \eta_C) \boxtimes (\kappa_D^{-1} \otimes \eta_D)(\chi_1+\chi_2) \otimes \omega_S^{-1} $ is Hom-free. This sequence cannot be cyclically $Ext^1$-connected by Serre duality and Kodaira vanishing theorem.
\end{proof}

Therefore we get the following consequence about the Hochschild cohomologies of the orthogonal complements of our exceptional sequences.

\begin{cor}
Let $\mathcal{A}$ be the orthogonal complement of the exceptional collection $ \mathcal{O}, \kappa_C^{-1} \boxtimes \eta_D(\chi_1), \eta_C \boxtimes \kappa_D^{-1}(\chi_2), (\kappa_C^{-1} \otimes \eta_C) \boxtimes (\kappa_D^{-1} \otimes \eta_D)(\chi_1+\chi_2) $. Then we have $ HH^i(S) = HH^i(\mathcal{A})$, for $i=0,1,2 $, and $ HH^3(S) \subset HH^3(\mathcal{A}) $.
\end{cor}

\begin{table}
\caption{$\alpha(\mathcal{O}_C(E^{e_2,e_4,e_3,e_1}))$}
\begin{tabularx}{\linewidth}{Y|Y|Y|Y|Y|Y|Y|Y|Y|Y|Y|Y|Y|Y|Y|Y|Y} \hline
$\alpha$ & 0 & $e_1$ & $e_2$ & $e_1+e_2$ & $e_3$ & $e_1+e_3$ & $e_2+e_3$ & $e_1+e_2+e_3$ & $e_4$ & $e_1+e_4$ & $e_2+e_4$ & $e_1+e_2+e_4$ & $e_3+e_4$ & $e_1+e_3+e_4$ & $e_2+e_3+e_4$ & $e_1+e_2+e_3+e_4$ \tabularnewline
\hline
$0$ & 1 & 1 & 1 & 1 & 1 & 1 & 1 & 1 & 1 & 1 & 1 & 1 & 1 & 1 & 1 & 1 \tabularnewline
\hline
$e_1$ & 1 & 1 & 1 & 1 & -1 & -1 & -1 & -1 & 1 & 1 & 1 & 1 & -1 & -1 & -1 & -1 \tabularnewline
\hline
$e_2$ & 1 & 1 & 1 & 1 & 1 & 1 & 1 & 1 & 1 & 1 & 1 & 1 & 1 & 1 & 1 & 1 \tabularnewline
\hline
$e_1+e_2$ & 1 & 1 & 1 & 1 & -1 & -1 & -1 & -1 & 1 & 1 & 1 & 1 & -1 & -1 & -1 & -1 \tabularnewline
\hline
$e_3$ & 1 & 1 & 1 & 1 & 1 & 1 & 1 & 1 & 1 & 1 & 1 & 1 & 1 & 1 & 1 & 1 \tabularnewline
\hline
$e_1+e_3$ & 1 & 1 & 1 & 1 & -1 & -1 & -1 & -1 & 1 & 1 & 1 & 1 & -1 & -1 & -1 & -1 \tabularnewline
\hline
$e_2+e_3$ & 1 & 1 & 1 & 1 & 1 & 1 & 1 & 1 & 1 & 1 & 1 & 1 & 1 & 1 & 1 & 1 \tabularnewline
\hline
$e_1+e_2+e_3$ & 1 & 1 & 1 & 1 & -1 & -1 & -1 & -1 & 1 & 1 & 1 & 1 & -1 & -1 & -1 & -1 \tabularnewline
\hline
$e_4$ & 1 & 1 & 1 & 1 & 1 & 1 & 1 & 1 & 1 & 1 & 1 & 1 & 1 & 1 & 1 & 1 \tabularnewline
\hline
$e_1+e_4$ & 1 & 1 & 1 & 1 & -1 & -1 & -1 & -1 & 1 & 1 & 1 & 1 & -1 & -1 & -1 & -1 \tabularnewline
\hline
$e_2+e_4$ & 1 & 1 & 1 & 1 & 1 & 1 & 1 & 1 & 1 & 1 & 1 & 1 & 1 & 1 & 1 & 1 \tabularnewline
\hline
$e_1+e_2+e_4$ & 1 & 1 & 1 & 1 & -1 & -1 & -1 & -1 & 1 & 1 & 1 & 1 & -1 & -1 & -1 & -1 \tabularnewline
\hline
$e_3+e_4$ & 1 & 1 & 1 & 1 & 1 & 1 & 1 & 1 & 1 & 1 & 1 & 1 & 1 & 1 & 1 & 1 \tabularnewline
\hline
$e_1+e_3+e_4$ & 1 & 1 & 1 & 1 & -1 & -1 & -1 & -1 & 1 & 1 & 1 & 1 & -1 & -1 & -1 & -1 \tabularnewline
\hline
$e_2+e_3+e_4$ & 1 & 1 & 1 & 1 & 1 & 1 & 1 & 1 & 1 & 1 & 1 & 1 & 1 & 1 & 1 & 1 \tabularnewline
\hline
$e_1+e_2+e_3+e_4$ & 1 & 1 & 1 & 1 & -1 & -1 & -1 & -1 & 1 & 1 & 1 & 1 & -1 & -1 & -1 & -1 \tabularnewline
\hline
\end{tabularx}
\end{table}

\begin{table}
\caption{$\alpha(\mathcal{O}_C(E^{e_1,e_2,e,e_3}))$}
\begin{tabularx}{\linewidth}{Y|Y|Y|Y|Y|Y|Y|Y|Y|Y|Y|Y|Y|Y|Y|Y|Y} \hline
$\alpha$ & 0 & $e_1$ & $e_2$ & $e_1+e_2$ & $e_3$ & $e_1+e_3$ & $e_2+e_3$ & $e_1+e_2+e_3$ & $e_4$ & $e_1+e_4$ & $e_2+e_4$ & $e_1+e_2+e_4$ & $e_3+e_4$ & $e_1+e_3+e_4$ & $e_2+e_3+e_4$ & $e_1+e_2+e_3+e_4$ \\
\hline
$0$ & 1 & 1 & 1 & 1 & 1 & 1 & 1 & 1 & 1 & 1 & 1 & 1 & 1 & 1 & 1 & 1 \\
\hline
$e_1$ & 1 & 1 & 1 & 1 & 1 & 1 & 1 & 1 & 1 & 1 & 1 & 1 & 1 & 1 & 1 & 1 \\
\hline
$e_2$ & 1 & 1 & 1 & 1 & 1 & 1 & 1 & 1 & 1 & 1 & 1 & 1 & 1 & 1 & 1 & 1 \\
\hline
$e_1+e_2$ & 1 & 1 & 1 & 1 & 1 & 1 & 1 & 1 & 1 & 1 & 1 & 1 & 1 & 1 & 1 & 1 \\
\hline
$e_3$ & 1 & 1 & 1 & 1 & 1 & 1 & 1 & 1 & -1 & -1 & -1 & -1 & -1 & -1 & -1 & -1 \\
\hline
$e_1+e_3$ & 1 & 1 & 1 & 1 & 1 & 1 & 1 & 1 & -1 & -1 & -1 & -1 & -1 & -1 & -1 & -1 \\
\hline
$e_2+e_3$ & 1 & 1 & 1 & 1 & 1 & 1 & 1 & 1 & -1 & -1 & -1 & -1 & -1 & -1 & -1 & -1 \\
\hline
$e_1+e_2+e_3$ & 1 & 1 & 1 & 1 & 1 & 1 & 1 & 1 & -1 & -1 & -1 & -1 & -1 & -1 & -1 & -1 \\
\hline
$e_4$ & 1 & 1 & 1 & 1 & 1 & 1 & 1 & 1 & -1 & -1 & -1 & -1 & -1 & -1 & -1 & -1 \\
\hline
$e_1+e_4$ & 1 & 1 & 1 & 1 & 1 & 1 & 1 & 1 & -1 & -1 & -1 & -1 & -1 & -1 & -1 & -1 \\
\hline
$e_2+e_4$ & 1 & 1 & 1 & 1 & 1 & 1 & 1 & 1 & -1 & -1 & -1 & -1 & -1 & -1 & -1 & -1 \\
\hline
$e_1+e_2+e_4$ & 1 & 1 & 1 & 1 & 1 & 1 & 1 & 1 & -1 & -1 & -1 & -1 & -1 & -1 & -1 & -1 \\
\hline
$e_3+e_4$ & 1 & 1 & 1 & 1 & 1 & 1 & 1 & 1 & 1 & 1 & 1 & 1 & 1 & 1 & 1 & 1 \\
\hline
$e_1+e_3+e_4$ & 1 & 1 & 1 & 1 & 1 & 1 & 1 & 1 & 1 & 1 & 1 & 1 & 1 & 1 & 1 & 1 \\
\hline
$e_2+e_3+e_4$ & 1 & 1 & 1 & 1 & 1 & 1 & 1 & 1 & 1 & 1 & 1 & 1 & 1 & 1 & 1 & 1 \\
\hline
$e_1+e_2+e_3+e_4$ & 1 & 1 & 1 & 1 & 1 & 1 & 1 & 1 & 1 & 1 & 1 & 1 & 1 & 1 & 1 & 1 \\
\hline
\end{tabularx}
\end{table}

\begin{table}
\caption{$\alpha(\mathcal{O}_C(E^{e_2,e,e_4,e_1}))$}
\begin{tabularx}{\linewidth}{Y|Y|Y|Y|Y|Y|Y|Y|Y|Y|Y|Y|Y|Y|Y|Y|Y} \hline
$\alpha$ & 0 & $e_1$ & $e_2$ & $e_1+e_2$ & $e_3$ & $e_1+e_3$ & $e_2+e_3$ & $e_1+e_2+e_3$ & $e_4$ & $e_1+e_4$ & $e_2+e_4$ & $e_1+e_2+e_4$ & $e_3+e_4$ & $e_1+e_3+e_4$ & $e_2+e_3+e_4$ & $e_1+e_2+e_3+e_4$ \tabularnewline
\hline
$0$ & 1 & 1 & 1 & 1 & 1 & 1 & 1 & 1 & 1 & 1 & 1 & 1 & 1 & 1 & 1 & 1 \tabularnewline
\hline
$e_1$ & 1 & 1 & 1 & 1 & -1 & -1 & -1 & -1 & -1 & -1 & -1 & -1 & 1 & 1 & 1 & 1 \tabularnewline
\hline
$e_2$ & 1 & 1 & 1 & 1 & 1 & 1 & 1 & 1 & 1 & 1 & 1 & 1 & 1 & 1 & 1 & 1 \tabularnewline
\hline
$e_1+e_2$ & 1 & 1 & 1 & 1 & -1 & -1 & -1 & -1 & -1 & -1 & -1 & -1 & 1 & 1 & 1 & 1 \tabularnewline
\hline
$e_3$ & 1 & 1 & 1 & 1 & -1 & -1 & -1 & -1 & -1 & -1 & -1 & -1 & 1 & 1 & 1 & 1 \tabularnewline
\hline
$e_1+e_3$ & 1 & 1 & 1 & 1 & 1 & 1 & 1 & 1 & 1 & 1 & 1 & 1 & 1 & 1 & 1 & 1 \tabularnewline
\hline
$e_2+e_3$ & 1 & 1 & 1 & 1 & -1 & -1 & -1 & -1 & -1 & -1 & -1 & -1 & 1 & 1 & 1 & 1 \tabularnewline
\hline
$e_1+e_2+e_3$ & 1 & 1 & 1 & 1 & 1 & 1 & 1 & 1 & 1 & 1 & 1 & 1 & 1 & 1 & 1 & 1 \tabularnewline
\hline
$e_4$ & 1 & 1 & 1 & 1 & 1 & 1 & 1 & 1 & 1 & 1 & 1 & 1 & 1 & 1 & 1 & 1 \tabularnewline
\hline
$e_1+e_4$ & 1 & 1 & 1 & 1 & -1 & -1 & -1 & -1 & -1 & -1 & -1 & -1 & 1 & 1 & 1 & 1 \tabularnewline
\hline
$e_2+e_4$ & 1 & 1 & 1 & 1 & 1 & 1 & 1 & 1 & 1 & 1 & 1 & 1 & 1 & 1 & 1 & 1 \tabularnewline
\hline
$e_1+e_2+e_4$ & 1 & 1 & 1 & 1 & -1 & -1 & -1 & -1 & -1 & -1 & -1 & -1 & 1 & 1 & 1 & 1 \tabularnewline
\hline
$e_3+e_4$ & 1 & 1 & 1 & 1 & -1 & -1 & -1 & -1 & -1 & -1 & -1 & -1 & 1 & 1 & 1 & 1 \tabularnewline
\hline
$e_1+e_3+e_4$ & 1 & 1 & 1 & 1 & 1 & 1 & 1 & 1 & 1 & 1 & 1 & 1 & 1 & 1 & 1 & 1 \tabularnewline
\hline
$e_2+e_3+e_4$ & 1 & 1 & 1 & 1 & -1 & -1 & -1 & -1 & -1 & -1 & -1 & -1 & 1 & 1 & 1 & 1 \tabularnewline
\hline
$e_1+e_2+e_3+e_4$ & 1 & 1 & 1 & 1 & 1 & 1 & 1 & 1 & 1 & 1 & 1 & 1 & 1 & 1 & 1 & 1 \tabularnewline
\hline
\end{tabularx}
\end{table}

\begin{table}
\caption{$\alpha(\mathcal{O}_D(F^{e+e_1,e_3+e_4,e_2+e_4,e_1+e_3}))$}
\begin{tabularx}{\linewidth}{Y|Y|Y|Y|Y|Y|Y|Y|Y|Y|Y|Y|Y|Y|Y|Y|Y} \hline
$\alpha$ & 0 & $e_1$ & $e_2$ & $e_1+e_2$ & $e_3$ & $e_1+e_3$ & $e_2+e_3$ & $e_1+e_2+e_3$ & $e_4$ & $e_1+e_4$ & $e_2+e_4$ & $e_1+e_2+e_4$ & $e_3+e_4$ & $e_1+e_3+e_4$ & $e_2+e_3+e_4$ & $e_1+e_2+e_3+e_4$ \tabularnewline
\hline
$0$ & 1 & 1 & 1 & 1 & 1 & 1 & 1 & 1 & 1 & 1 & 1 & 1 & 1 & 1 & 1 & 1 \tabularnewline
\hline
$e_1$ & 1 & -1 & 1 & -1 & -1 & 1 & -1 & 1 & -1 & 1 & -1 & 1 & 1 & -1 & 1 & -1 \tabularnewline
\hline
$e_2$ & 1 & 1 & 1 & 1 & 1 & 1 & 1 & 1 & 1 & 1 & 1 & 1 & 1 & 1 & 1 & 1 \tabularnewline
\hline
$e_1+e_2$ & 1 & -1 & 1 & -1 & -1 & 1 & -1 & 1 & -1 & 1 & -1 & 1 & 1 & -1 & 1 & -1 \tabularnewline
\hline
$e_3$ & 1 & 1 & 1 & 1 & 1 & 1 & 1 & 1 & 1 & 1 & 1 & 1 & 1 & 1 & 1 & 1 \tabularnewline
\hline
$e_1+e_3$& 1 & -1 & 1 & -1 & -1 & 1 & -1 & 1 & -1 & 1 & -1 & 1 & 1 & -1 & 1 & -1 \tabularnewline
\hline
$e_2+e_3$ & 1 & 1 & 1 & 1 & 1 & 1 & 1 & 1 & 1 & 1 & 1 & 1 & 1 & 1 & 1 & 1 \tabularnewline
\hline
$e_1+e_2+e_3$& 1 & -1 & 1 & -1 & -1 & 1 & -1 & 1 & -1 & 1 & -1 & 1 & 1 & -1 & 1 & -1 \tabularnewline
\hline
$e_4$ & 1 & 1 & 1 & 1 & 1 & 1 & 1 & 1 & 1 & 1 & 1 & 1 & 1 & 1 & 1 & 1 \tabularnewline
\hline
$e_1+e_4$& 1 & -1 & 1 & -1 & -1 & 1 & -1 & 1 & -1 & 1 & -1 & 1 & 1 & -1 & 1 & -1 \tabularnewline
\hline
$e_2+e_4$ & 1 & 1 & 1 & 1 & 1 & 1 & 1 & 1 & 1 & 1 & 1 & 1 & 1 & 1 & 1 & 1 \tabularnewline
\hline
$e_1+e_2+e_4$& 1 & -1 & 1 & -1 & -1 & 1 & -1 & 1 & -1 & 1 & -1 & 1 & 1 & -1 & 1 & -1 \tabularnewline
\hline
$e_3+e_4$ & 1 & 1 & 1 & 1 & 1 & 1 & 1 & 1 & 1 & 1 & 1 & 1 & 1 & 1 & 1 & 1 \tabularnewline
\hline
$e_1+e_3+e_4$& 1 & -1 & 1 & -1 & -1 & 1 & -1 & 1 & -1 & 1 & -1 & 1 & 1 & -1 & 1 & -1 \tabularnewline
\hline
$e_2+e_3+e_4$ & 1 & 1 & 1 & 1 & 1 & 1 & 1 & 1 & 1 & 1 & 1 & 1 & 1 & 1 & 1 & 1 \tabularnewline
\hline
$e_1+e_2+e_3+e_4$& 1 & -1 & 1 & -1 & -1 & 1 & -1 & 1 & -1 & 1 & -1 & 1 & 1 & -1 & 1 & -1 \tabularnewline
\hline
\end{tabularx}
\end{table}

\begin{table}
\caption{$\alpha(\mathcal{O}_D(F^{e_1+e_3,e_2+e_4,e+e_1,e+e_2}))$}
\begin{tabularx}{\linewidth}{Y|Y|Y|Y|Y|Y|Y|Y|Y|Y|Y|Y|Y|Y|Y|Y|Y} \hline
$\alpha$ & 0 & $e_1$ & $e_2$ & $e_1+e_2$ & $e_3$ & $e_1+e_3$ & $e_2+e_3$ & $e_1+e_2+e_3$ & $e_4$ & $e_1+e_4$ & $e_2+e_4$ & $e_1+e_2+e_4$ & $e_3+e_4$ & $e_1+e_3+e_4$ & $e_2+e_3+e_4$ & $e_1+e_2+e_3+e_4$ \tabularnewline
\hline
$0$ & 1 & 1 & 1 & 1 & 1 & 1 & 1 & 1 & 1 & 1 & 1 & 1 & 1 & 1 & 1 & 1 \tabularnewline
\hline
$e_1$ & 1 & 1 & 1 & 1 & 1 & 1 & 1 & 1 & 1 & 1 & 1 & 1 & 1 & 1 & 1 & 1 \tabularnewline
\hline
$e_2$ & 1 & -1 & 1 & -1 & -1 & 1 & -1 & 1 & 1 & -1 & 1 & -1 & -1 & 1 & -1 & 1 \tabularnewline
\hline
$e_1+e_2$ & 1 & -1 & 1 & -1 & -1 & 1 & -1 & 1 & 1 & -1 & 1 & -1 & -1 & 1 & -1 & 1 \tabularnewline
\hline
$e_3$ & 1 & 1 & 1 & 1 & 1 & 1 & 1 & 1 & 1 & 1 & 1 & 1 & 1 & 1 & 1 & 1 \tabularnewline
\hline
$e_1+e_3$ & 1 & 1 & 1 & 1 & 1 & 1 & 1 & 1 & 1 & 1 & 1 & 1 & 1 & 1 & 1 & 1 \tabularnewline
\hline
$e_2+e_3$ & 1 & -1 & 1 & -1 & -1 & 1 & -1 & 1 & 1 & -1 & 1 & -1 & -1 & 1 & -1 & 1 \tabularnewline
\hline
$e_1+e_2+e_3$ & 1 & -1 & 1 & -1 & -1 & 1 & -1 & 1 & 1 & -1 & 1 & -1 & -1 & 1 & -1 & 1 \tabularnewline
\hline
$e_4$ & 1 & -1 & 1 & -1 & -1 & 1 & -1 & 1 & 1 & -1 & 1 & -1 & -1 & 1 & -1 & 1 \tabularnewline
\hline
$e_1+e_4$ & 1 & -1 & 1 & -1 & -1 & 1 & -1 & 1 & 1 & -1 & 1 & -1 & -1 & 1 & -1 & 1 \tabularnewline
\hline
$e_2+e_4$ & 1 & 1 & 1 & 1 & 1 & 1 & 1 & 1 & 1 & 1 & 1 & 1 & 1 & 1 & 1 & 1 \tabularnewline
\hline
$e_1+e_2+e_4$ & 1 & 1 & 1 & 1 & 1 & 1 & 1 & 1 & 1 & 1 & 1 & 1 & 1 & 1 & 1 & 1 \tabularnewline
\hline
$e_3+e_4$ & 1 & -1 & 1 & -1 & -1 & 1 & -1 & 1 & 1 & -1 & 1 & -1 & -1 & 1 & -1 & 1 \tabularnewline
\hline
$e_1+e_3+e_4$ & 1 & -1 & 1 & -1 & -1 & 1 & -1 & 1 & 1 & -1 & 1 & -1 & -1 & 1 & -1 & 1 \tabularnewline
\hline
$e_2+e_3+e_4$ & 1 & 1 & 1 & 1 & 1 & 1 & 1 & 1 & 1 & 1 & 1 & 1 & 1 & 1 & 1 & 1 \tabularnewline
\hline
$e_1+e_2+e_3+e_4$ & 1 & 1 & 1 & 1 & 1 & 1 & 1 & 1 & 1 & 1 & 1 & 1 & 1 & 1 & 1 & 1 \tabularnewline
\hline
\end{tabularx}
\end{table}

\begin{table}
\caption{$\alpha(\mathcal{O}_D(F^{e+e_2,e_2+e_4,e_1+e_3,e_3+e_4}))$}
\begin{tabularx}{\linewidth}{Y|Y|Y|Y|Y|Y|Y|Y|Y|Y|Y|Y|Y|Y|Y|Y|Y} \hline
$\alpha$ & 0 & $e_1$ & $e_2$ & $e_1+e_2$ & $e_3$ & $e_1+e_3$ & $e_2+e_3$ & $e_1+e_2+e_3$ & $e_4$ & $e_1+e_4$ & $e_2+e_4$ & $e_1+e_2+e_4$ & $e_3+e_4$ & $e_1+e_3+e_4$ & $e_2+e_3+e_4$ & $e_1+e_2+e_3+e_4$ \tabularnewline
\hline
$0$ & 1 & 1 & 1 & 1 & 1 & 1 & 1 & 1 & 1 & 1 & 1 & 1 & 1 & 1 & 1 & 1 \tabularnewline
\hline
$e_1$ & 1 & 1 & -1 & -1 & -1 & -1 & 1 & 1 & -1 & -1 & 1 & 1 & 1 & 1 & -1 & -1 \tabularnewline
\hline
$e_2$ & 1 & 1 & 1 & 1 & 1 & 1 & 1 & 1 & 1 & 1 & 1 & 1 & 1 & 1 & 1 & 1 \tabularnewline
\hline
$e_1+e_2$ & 1 & 1 & -1 & -1 & -1 & -1 & 1 & 1 & -1 & -1 & 1 & 1 & 1 & 1 & -1 & -1 \tabularnewline
\hline
$e_3$ & 1 & 1 & -1 & -1 & -1 & -1 & 1 & 1 & -1 & -1 & 1 & 1 & 1 & 1 & -1 & -1 \tabularnewline
\hline
$e_1+e_3$ & 1 & 1 & 1 & 1 & 1 & 1 & 1 & 1 & 1 & 1 & 1 & 1 & 1 & 1 & 1 & 1 \tabularnewline
\hline
$e_2+e_3$ & 1 & 1 & -1 & -1 & -1 & -1 & 1 & 1 & -1 & -1 & 1 & 1 & 1 & 1 & -1 & -1 \tabularnewline
\hline
$e_1+e_2+e_3$ & 1 & 1 & 1 & 1 & 1 & 1 & 1 & 1 & 1 & 1 & 1 & 1 & 1 & 1 & 1 & 1 \tabularnewline
\hline
$e_4$ & 1 & 1 & 1 & 1 & 1 & 1 & 1 & 1 & 1 & 1 & 1 & 1 & 1 & 1 & 1 & 1 \tabularnewline
\hline
$e_1+e_4$ & 1 & 1 & -1 & -1 & -1 & -1 & 1 & 1 & -1 & -1 & 1 & 1 & 1 & 1 & -1 & -1 \tabularnewline
\hline
$e_2+e_4$ & 1 & 1 & 1 & 1 & 1 & 1 & 1 & 1 & 1 & 1 & 1 & 1 & 1 & 1 & 1 & 1 \tabularnewline
\hline
$e_1+e_2+e_4$ & 1 & 1 & -1 & -1 & -1 & -1 & 1 & 1 & -1 & -1 & 1 & 1 & 1 & 1 & -1 & -1 \tabularnewline
\hline
$e_3+e_4$ & 1 & 1 & -1 & -1 & -1 & -1 & 1 & 1 & -1 & -1 & 1 & 1 & 1 & 1 & -1 & -1 \tabularnewline
\hline
$e_1+e_3+e_4$ & 1 & 1 & 1 & 1 & 1 & 1 & 1 & 1 & 1 & 1 & 1 & 1 & 1 & 1 & 1 & 1 \tabularnewline
\hline
$e_2+e_3+e_4$ & 1 & 1 & -1 & -1 & -1 & -1 & 1 & 1 & -1 & -1 & 1 & 1 & 1 & 1 & -1 & -1 \tabularnewline
\hline
$e_1+e_2+e_3+e_4$ & 1 & 1 & 1 & 1 & 1 & 1 & 1 & 1 & 1 & 1 & 1 & 1 & 1 & 1 & 1 & 1 \tabularnewline
\hline
\end{tabularx}
\end{table}

\begin{table}
\caption{$\alpha( \kappa_C \boxtimes \eta_D) $}
\begin{tabularx}{\linewidth}{Y|Y|Y|Y|Y|Y|Y|Y|Y|Y|Y|Y|Y|Y|Y|Y|Y} \hline
$\alpha$ & 0 & $e_1$ & $e_2$ & $e_1+e_2$ & $e_3$ & $e_1+e_3$ &
$e_2+e_3$ & $e_1+e_2+e_3$ & $e_4$ & $e_1+e_4$ & $e_2+e_4$ &
$e_1+e_2+e_4$ & $e_3+e_4$ & $e_1+e_3+e_4$ & $e_2+e_3+e_4$ &
$e_1+e_2+e_3+e_4$ \tabularnewline \hline $0$ & 1 & 1 & 1 & 1 & 1 & 1
& 1 & 1 & 1 & 1 & 1 & 1 & 1 & 1 & 1 & 1 \tabularnewline \hline $e_1$
& 1 & 1 & -1 & -1 & 1 & 1 & -1 & -1 & -1 & -1 & 1 & 1 & -1 & -1 & 1
& 1 \tabularnewline \hline $e_2$ & 1 & -1 & 1 & -1 & -1 & 1 & -1 & 1
& 1 & -1 & 1 & -1 & -1 & 1 & -1 & 1 \tabularnewline \hline $e_1+e_2$
& 1 & -1 & -1 & 1 & -1 & 1 & 1 & -1 & -1 & 1 & 1 & -1 & 1 & -1 & -1
& 1 \tabularnewline \hline $e_3$ & 1 & 1 & -1 & -1 & -1 & -1 & 1 & 1
& -1 & -1 & 1 & 1 & 1 & 1 & -1 & -1 \tabularnewline \hline $e_1+e_3$
& 1 & 1 & 1 & 1 & -1 & -1 & -1 & -1 & 1 & 1 & 1 & 1 & -1 & -1 & -1 &
-1 \tabularnewline \hline $e_2+e_3$ & 1 & -1 & -1 & 1 & 1 & -1 & -1
& 1 & -1 & 1 & 1 & -1 & -1 & 1 & 1 & -1 \tabularnewline \hline
$e_1+e_2+e_3$ & 1 & -1 & 1 & -1 & 1 & -1 & 1 & -1 & 1 & -1 & 1 & -1
& 1 & -1 & 1 & -1 \tabularnewline \hline $e_4$ & 1 & -1 & 1 & -1 &
-1 & 1 & -1 & 1 & 1 & -1 & 1 & -1 & -1 & 1 & -1 & 1 \tabularnewline
\hline $e_1+e_4$ & 1 & -1 & -1 & 1 & -1 & 1 & 1 & -1 & -1 & 1 & 1 &
-1 & 1 & -1 & -1 & 1 \tabularnewline \hline $e_2+e_4$ & 1 & 1 & 1 &
1 & 1 & 1 & 1 & 1 & 1 & 1 & 1 & 1 & 1 & 1 & 1 & 1 \tabularnewline
\hline $e_1+e_2+e_4$ & 1 & 1 & -1 & -1 & 1 & 1 & -1 & -1 & -1 & -1 &
1 & 1 & -1 & -1 & 1 & 1 \tabularnewline \hline $e_3+e_4$ & 1 & -1 &
-1 & 1 & 1 & -1 & -1 & 1 & -1 & 1 & 1 & -1 & -1 & 1 & 1 & -1
\tabularnewline \hline $e_1+e_3+e_4$ & 1 & -1 & 1 & -1 & 1 & -1 & 1
& -1 & 1 & -1 & 1 & -1 & 1 & -1 & 1 & -1 \tabularnewline \hline
$e_2+e_3+e_4$ & 1 & 1 & -1 & -1 & -1 & -1 & 1 & 1 & -1 & -1 & 1 & 1
& 1 & 1 & -1 & -1 \tabularnewline \hline $e_1+e_2+e_3+e_4$ & 1 & 1 &
1 & 1 & -1 & -1 & -1 & -1 & 1 & 1 & 1 & 1 & -1 & -1 & -1 & -1
\tabularnewline \hline
\end{tabularx}
\end{table}

\begin{table}
\caption{$\alpha(\eta_C \boxtimes \kappa_D) $}
\begin{tabularx}{\linewidth}{Y|Y|Y|Y|Y|Y|Y|Y|Y|Y|Y|Y|Y|Y|Y|Y|Y} \hline
$\alpha$ & 0 & $e_1$ & $e_2$ & $e_1+e_2$ & $e_3$ & $e_1+e_3$ & $e_2+e_3$ & $e_1+e_2+e_3$ & $e_4$ & $e_1+e_4$ & $e_2+e_4$ & $e_1+e_2+e_4$ & $e_3+e_4$ & $e_1+e_3+e_4$ & $e_2+e_3+e_4$ & $e_1+e_2+e_3+e_4$ \tabularnewline
\hline
$0$ & 1 & 1 & 1 & 1 & 1 & 1 & 1 & 1 & 1 & 1 & 1 & 1 & 1 & 1 & 1 & 1 \tabularnewline
\hline
$e_1$ & 1 & -1 & 1 & -1 & 1 & -1 & 1 & -1 & 1 & -1 & 1 & -1 & 1 & -1 & 1 & -1 \tabularnewline
\hline
$e_2$ & 1 & 1 & 1 & 1 & 1 & 1 & 1 & 1 & 1 & 1 & 1 & 1 & 1 & 1 & 1 & 1 \tabularnewline
\hline
$e_1+e_2$ & 1 & -1 & 1 & -1 & 1 & -1 & 1 & -1 & 1 & -1 & 1 & -1 & 1 & -1 & 1 & -1 \tabularnewline
\hline
$e_3$ & 1 & 1 & 1 & 1 & -1 & -1 & -1 & -1 & 1 & 1 & 1 & 1 & -1 & -1 & -1 & -1 \tabularnewline
\hline
$e_1+e_3$ & 1 & -1 & 1 & -1 & -1 & 1 & -1 & 1 & 1 & -1 & 1 & -1 & -1 & 1 & -1 & 1 \tabularnewline
\hline
$e_2+e_3$ & 1 & 1 & 1 & 1 & -1 & -1 & -1 & -1 & 1 & 1 & 1 & 1 & -1 & -1 & -1 & -1 \tabularnewline
\hline
$e_1+e_2+e_3$ & 1 & -1 & 1 & -1 & -1 & 1 & -1 & 1 & 1 & -1 & 1 & -1 & -1 & 1 & -1 & 1 \tabularnewline
\hline
$e_4$ & 1 & 1 & 1 & 1 & 1 & 1 & 1 & 1 & -1 & -1 & -1 & -1 & -1 & -1 & -1 & -1 \tabularnewline
\hline
$e_1+e_4$ & 1 & -1 & 1 & -1 & 1 & -1 & 1 & -1 & -1 & 1 & -1 & 1 & -1 & 1 & -1 & 1 \tabularnewline
\hline
$e_2+e_4$ & 1 & 1 & 1 & 1 & 1 & 1 & 1 & 1 & -1 & -1 & -1 & -1 & -1 & -1 & -1 & -1 \tabularnewline
\hline
$e_1+e_2+e_4$ & 1 & -1 & 1 & -1 & 1 & -1 & 1 & -1 & -1 & 1 & -1 & 1 & -1 & 1 & -1 & 1 \tabularnewline
\hline
$e_3+e_4$ & 1 & 1 & 1 & 1 & -1 & -1 & -1 & -1 & -1 & -1 & -1 & -1 & 1 & 1 & 1 & 1 \tabularnewline
\hline
$e_1+e_3+e_4$ & 1 & -1 & 1 & -1 & -1 & 1 & -1 & 1 & -1 & 1 & -1 & 1 & 1 & -1 & 1 & -1 \tabularnewline
\hline
$e_2+e_3+e_4$ & 1 & 1 & 1 & 1 & -1 & -1 & -1 & -1 & -1 & -1 & -1 & -1 & 1 & 1 & 1 & 1 \tabularnewline
\hline
$e_1+e_2+e_3+e_4$ & 1 & -1 & 1 & -1 & -1 & 1 & -1 & 1 & -1 & 1 & -1 & 1 & 1 & -1 & 1 & -1 \tabularnewline
\hline
\end{tabularx}
\end{table}

\end{document}